\documentclass{amsart}

\usepackage{a4wide,latexsym,amssymb,amsthm,amsmath}
\usepackage{color}
\usepackage{pifont}
\usepackage{enumerate}

\usepackage{hyperref}

\theoremstyle{plain}

\newtheorem{theorem}{Theorem}
\newtheorem{lemma}{Lemma}

\theoremstyle{remark}

\newtheorem{remark}{Remark}
\newtheorem{example}{Example}
\def\R{\mathbb{R}}

\newcommand\A{{\mathfrak{a}}}
\newcommand\B{{\mathfrak{b}}}
\newcommand\g{{\mathfrak{g}}}
\newcommand\h{{\mathfrak{h}}}
\newcommand\m{{\mathfrak{m}}}

\newcommand{\xmark}{\ding{55}}

\newcommand\End{\operatorname{End}}

\newcommand\Ad{\operatorname{Ad}}
\newcommand\ad{\operatorname{ad}}

\newcommand\sd{\mathrm{d}}

\begin{document}

\title[] {Parallel and totally geodesic hypersurfaces of non-reductive homogeneous four-manifolds}

\author[G.~Calvaruso]{Giovanni Calvaruso}
\address{Universit\`a del Salento\\ Dipartimento di Matematica ``E.~De Giorgi''\\  Provinciale Lecce--Arnesano\\ 73100 Lecce\\ Italy} \email{giovanni.calvaruso@unile.it}

\author[R. Storm]{Reinier Storm}
\address{KU Leuven\\ Department of Mathematics\\ Celestijnenlaan 200B -- Box 2400\\ BE-3001 Leuven\\ Belgium} \email{reinier.storm@kuleuven.be}

\author[J. Van der Veken]{Joeri Van der Veken}
\address{KU Leuven\\ Department of Mathematics\\ Celestijnenlaan 200B -- Box 2400\\ BE-3001 Leuven\\ Belgium} \email{joeri.vanderveken@kuleuven.be}

\thanks{The first author is partially supported by funds of the University of Salento and GNSAGA. The second and the third author are supported by project 3E160361 of the KU Leuven Research Fund and the third author is supported by EOS project G0H4518N of the Belgian government.}

\begin{abstract} 
We classify totally geodesic and parallel hypersurfaces of four-dimensional non-reductive homogeneous pseudo-Riemannian manifolds.
\end{abstract}



\maketitle


\section{Introduction}
A pseudo-Riemannian manifold $(M,g)$ is said to be \textit{homogeneous} if there exists a  connected Lie group of isometries acting transitively on it. 
It is well-known that any such manifold can be then realized as a coset space $M=G/H$, with $H$ denoting the isotropy subgroup at a point chosen as origin. 

We say that a homogeneous pseudo-Riemannian manifold $(M,g)$ is \emph{reductive} if there exists a connected transitive group of isometries $G\subseteq \mathrm{Iso}(M,g)$ for which $\g =\h \oplus \m$, where $\g$ is the Lie algebra of~$G$ and $\h$ is the Lie algebra of the isotropy subgroup $H\subseteq G$, such that $\m$ is an $\Ad(H)$-invariant subspace of ~$\g$. When $H$ is connected, the last condition is equivalent to the algebraic condition $[\h ,\m ]\subseteq \m$. In such a case, we call $\m$  a \emph{reductive complement} of $\h$ and $\h \oplus \m$ a {\em reductive decomposition}. A homogeneous pseudo-Riemannian manifold $(M,g)$ is said to be {\em non-reductive} if there exists a connected Lie group $G$ acting transitively on $(M,g)$, such that the corresponding Lie algebra $\g$ does not admit {\em any} reductive decomposition, that is, for any complement~$\m$ of $\h$ in $\g$ one has $[\h,\m]\nsubseteq \m$ (see  \cite{FR}). 

A homogeneous {\em Riemannian} manifold is necessarily reductive. Moreover, there do not exist any two- and three-dimensional non-reductive homogeneous pseudo-Riemannian manifolds \cite{FR}. On the other hand, there exist four-dimensional non-reductive homogeneous pseudo-Riemannian manifolds, both of Lorentzian and of neutral signature. After their classification, given in \cite{FR}, the geometry of these spaces has been investigated from different points of view, showing several interesting behaviours. In \cite{The}, invariant Yang-Mills connections over these spaces have been classified. More recently, their curvature properties \cite{CF}, Walker structures \cite{CZ}, global coordinates and homogeneous geodesics \cite{CFZ}, Ricci solitons \cite{CZ2}, symmetries \cite{CZ3} and conformal geometry \cite{GRR} were studied.
 
In this paper we address the problem of classifying parallel and, in particular, totally geodesic hypersurfaces of four-dimensional non-reductive homogeneous spaces. The search for parallel submanifolds of a given pseudo-Riemannian manifold is a natural problem, which enriches our knowledge and understanding of the geometry of the manifold itself. \textit{Parallel} submanifolds are those for which the second fundamental form, and hence all the extrinsic invariants derived from it, are covariantly constant. This condition is a natural generalization of \textit{totally geodesic} submanifolds, for which the second fundamental form vanishes identically. The latter condition means that the geodesics of the submanifold are also geodesics of the ambient space. 

It is worthwhile to recall that parallel submanifolds of a locally symmetric ambient space are locally symmetric, but, for general ambient spaces, parallel submanifolds need not be locally symmetric. For this reason, it is interesting to investigate parallel hypersurfaces of homogeneous spaces which are not locally symmetric. 

Parallel surfaces in three-dimensional homogeneous spaces have been intensively studied, showing several interesting behaviours (see for example \cite{CVdV1,CVdV2,CVdV4,CVdV5}). In higher dimensions, the condition for the existence of parallel hypersurfaces in non-symmetric ambient spaces becomes more problematic. For instance, among the four-dimensional generalized symmetric spaces, the only ones admitting such hypersurfaces are those of type $\textbf{C}$, which are in fact symmetric \cite{DdV}.

The results we obtained about the existence of parallel and totally geodesic hypersurfaces in non-reductive, not locally symmetric homogeneous pseudo-Riemannian four-manifolds are summarized in the following Table~I. In the first column we listed the six different types of non-reductive (not locally symmetric) homogeneous four-manifolds. The marks in the other columns show whether such a space admits totally geodesic hypersurfaces (the second fundamental form vanishes identically), \textit{proper parallel} hypersurfaces (the covariant derivative of the second fundamental form vanishes, but the second fundamental form itself does not) and hypersurfaces with a \textit{Codazzi second fundamental form} (the covariant derivative of the second fundamental form is totally symmetric). We may observe that different behaviours occur, and they are not influenced by the signature of the invariant metrics, so they are rather related to the different structures of the Lie algebra $\g$ of the group of isometries.
\medskip
\begin{center}
	\begin{tabular}{|l|c|c|c|}
		\hline
		Ambient space $\vphantom{A^{A^A}}$ & Codazzi second fundamental form & Proper parallel & Totally geodesic \\
		\hline
		\quad type $\textbf{A1}$ $\vphantom{A^{A^A}}$ & \xmark & \xmark & \xmark  \\[2 pt]
		\hline
		\quad type $\textbf{A2}$ $\vphantom{A^{A^A}}$ & \checkmark & \checkmark & \checkmark \\[2 pt]
		\hline
		\quad type $\textbf{A3}$ $\vphantom{A^{A^A}}$ & \checkmark & \checkmark & \checkmark \\[2 pt]\hline
		\quad type $\textbf{A4}$ $\vphantom{A^{A^A}}$ & \checkmark & \checkmark & \checkmark  \\[2 pt]\hline
		\quad type $\textbf{B1}$ $\vphantom{A^{A^A}}$ & \xmark & \xmark & \xmark  \\[2 pt]\hline
		\quad type $\textbf{B2}$ $\vphantom{A^{A^A}}$ & \checkmark & \checkmark & \checkmark  \\[2 pt]\hline
	\end{tabular} 
\\ $\vphantom{A^{A^A}}${\em Table I: parallel and totally geodesic hypersurfaces of non-reductive  \\ homogeneous pseudo-Riemannian $4$-manifolds}\end{center}
\bigskip

The paper is organized in the following way. In Section~\ref{sec:preliminaries}, we recall the classification of non-reductive homogeneous pseudo-Riemannian four-manifolds as well as the basics on parallel hypersurfaces. The needed information about the Levi-Civita connection and the curvature of non-reductive homogeneous pseudo-Riemannian four-manifolds is discussed in Section~\ref{sec:LC connection}, where vector fields are treated as functions from the overlying Lie group to its Lie algebra. In the remaining sections we classify parallel and totally geodesic hypersurfaces of each of the six types.


\section{Preliminaries} \label{sec:preliminaries}

\subsection{Non-reductive homogeneous four-manifolds}
\setcounter{equation}{0}

Four-dimensional non-reductive homogeneous pseudo-Riemannian manifolds $M=G/H$ were classified in \cite{FR} in terms of the corresponding Lie algebras $(\g,\h)$, with $\h \subseteq \g$. For a fixed Lie algebra pair from the classification in \cite{FR} one can compute all invariant pseudo-Riemannian metrics. We report below their description, obtained following the same argument already used in \cite{CF}, but with some different choices on the complement of the isotropy subalgebra, which make some computations easier. 

\medskip\noindent
{\bf Type A1.} Let $\g=\A _1$ be the decomposable $5$-dimensional Lie algebra $\mathfrak{sl}(2,\R) \oplus \mathfrak{s} (2)$, where $\mathfrak{s} (2)$ is the $2$-dimensional solvable algebra. There exists a basis $\{e_1,\ldots,e_5\}$ of $\g$, such that the non-zero brackets are
\[
[e_1,e_2]=2e_2, \quad [e_1,e_3]=-2e_3, \quad [e_2,e_3]=e_1, \quad [e_4,e_5]=e_4
\]
and the isotropy algebra is $\h={\rm span}\{h_1=e_3+e_4\}$. So we can take
\begin{equation}\label{eq:mA1}
\m= {\rm span}\{u_1=e_1,u_2=e_2,u_3=e_5,u_4=2 e_3\}.
\end{equation}
as a complement of $\h$ in $\g$. With respect to the basis $\{\theta^1,\ldots,\theta^4\}$, dual to $\{u_1,\ldots,u_4\}$, we have the following description of non-degenerate invariant metrics:
\begin{equation}\label{gA1}
\begin{array}{l} g=   a \,\left(\theta^1 \circ \theta^1 - \theta^1 \circ \theta^3+2 \theta^2 \circ \theta^4\right) +b\, \theta^2 \circ \theta^2 + 2c \, \theta^2 \circ \theta^3+d \, \theta^3 \circ \theta^3,  \quad a ( a -4d)\neq 0,
\end{array}
\end{equation}
where $a$, $b$, $c$ and $d$ are real constants and $\circ$ denotes the symmetric product of $1$-forms, i.e., $\theta^i \circ \theta^j = \frac{1}{2}(\theta^i\otimes\theta^j + \theta^j\otimes\theta^i)$. Depending on the sign of
 $a ( a -4d)\neq 0$, the metric is either Lorentzian or it has neutral signature.

\medskip\noindent
{\bf Type A2.} Let $\g =\A _2$ be the one-parameter family of $5$-dimensional Lie algebras $A_{5,30}$ from \cite{PSW}. There exists a basis $\{e_1,\ldots,e_5\}$ of $\g$, such that the non-zero brackets are
\[
\begin{array}{lll}
[e_1,e_5]=(\kappa+1)e_1, \quad & [e_2,e_4]=e_1, \quad & [e_2,e_5]=\kappa e_2, \\[2 pt]
[e_3,e_4]=e_2, \quad & [e_3,e_5]=(\kappa-1) e_3, \quad  & [e_4,e_5]=e_4
\end{array}
\]
for a parameter $\kappa \in \R$ and the isotropy algebra is $\h={\rm span}\{h_1=e_4\}$. Hence, we can take
\begin{equation}\label{eq:mA2}
\m={\rm span}\{u_1=e_1,u_2=e_2,u_3=e_3,u_4=e_5\}
\end{equation}
and we find that the non-degenerate invariant metrics are described by
\begin{equation}\label{gA2}
\begin{array}{l} g=a\, \left(-2\theta^1 \circ \theta^3+\theta^2 \circ \theta^2 \right)+b\, \theta^3 \circ \theta^3 +2c\, \theta^3 \circ \theta^4 +d\, \theta^4 \circ \theta^4, \quad ad \neq 0,
\end{array}
\end{equation}
where $a$, $b$, $c$ and $d$ are real constants. Depending on the sign of
 $ad$, the metric is either Lorentzian or it has neutral signature.

\medskip\noindent
{\bf Type A3.} Let $\g=\A _3$ be one of the $5$-dimensional Lie algebras $A_{5,37}$ or $A_{5,36}$ from \cite{PSW}. There exists a basis $\{e_1,\ldots,e_5\}$ of $\g$, such that the non-zero brackets are
\[
\begin{array}{llllll}
[e_1,e_4]=2e_1, & [e_2,e_3]=e_1, & [e_2,e_4]=e_2,  &
[e_2,e_5]=-\eta e_3,  & [e_3,e_4]=e_3, & [e_3,e_5]=e_2 ,
\end{array}
\]
with $\eta =1$ for $A_{5,37}$ and $\eta =-1$ for $A_{5,36}$, and the isotropy algebra is $\h={\rm span}\{h_1=e_3\}$. Thus, we can take
\begin{equation}\label{eq:mA3}
\m={\rm span}\{u_1=e_1,u_2=e_2+e_3,u_3=e_4,u_4=e_5\}
\end{equation}
and the non-degenerate invariant metrics are
\begin{equation}\label{gA3}
\begin{array}{l} g=a\, \left(2\theta^1 \circ \theta^4+\theta^2 \circ \theta^2 \right)+b\, \theta^3 \circ \theta^3 +2c\, \theta^3 \circ \theta^4 +d\, \theta^4 \circ \theta^4, \quad ab \neq 0,
\end{array}
\end{equation}
where $a$, $b$, $c$ and $d$ are real constants. Depending on the sign of
 $ab$, the metric is either Lorentzian or it has neutral signature.


\medskip\noindent
{\bf Types A4 and B2.} Let $\g=\A _4$ be the $6$-dimensional Schroedinger Lie algebra  $\mathfrak{sl}(2,\R) \ltimes \mathfrak{n} (3)$, where $\mathfrak{n} (3)$ is the $3$-dimensional Heisenberg algebra. There exists a basis $\{e_1,\ldots,e_6\}$ of $\g$, such that the non-zero brackets are
\[
\begin{array}{llll}
[e_1,e_2]=2e_2, \quad & [e_1,e_3]=-2e_3, \quad & [e_2,e_3]=e_1, \quad & [e_1,e_4]=e_4, \\[2 pt]
[e_1,e_5]=-e_5, \quad & [e_2,e_5]=e_4, \quad  & [e_3,e_4]=e_5 , \quad & [e_4,e_5]=e_6.
\end{array}
\]
For type \textbf{A4}, the isotropy algebra is $\h={\rm span}\{h_1=e_3+e_6, \, h_2=e_5\}$ and for type \textbf{B2} it is $\h={\rm span}\{h_1={e_3 - e_6}, \, h_2={e_5}\}$. Therefore, in both cases, we can take
\begin{equation}\label{eq:mA4&B2}
\m={\rm span}\{u_1 = e_1, u_2 = e_2, u_3 = -2e_6, u_4 = e_4\}.
\end{equation}
The non-degenerate invariant metrics are given by
\begin{equation}\label{gA4}
\begin{array}{l} g=a \left(\eta \, \theta^1 \circ \theta^1+2 \, \theta^2 \circ \theta^3+\frac 12 \theta^4 \circ \theta ^4 \right)+b\, \theta^2 \circ \theta^2, \quad a \neq 0,
\end{array} \end{equation}
where $a$ and $b$ are real constants, $\eta = 1$ for type \textbf{A4} and $\eta = -1$ for type \textbf{B2}. Remark that the metric has Lorentzian signature for $\eta = 1$, whereas it has neutral signature for $\eta = -1$. 

\medskip\noindent
{\bf Type B1.} Let $\g =\B_1$ be the $5$-dimensional Lie algebra  $\mathfrak{sl}(2,\R) \ltimes \R ^2 $. Then there exists a basis $\{e_1,\ldots,e_5\}$ for which the non-zero brackets are
\[
\begin{array}{llll}
[e_1,e_2]=2e_2, \quad & [e_1,e_3]=-2e_3, \quad & [e_2,e_3]=e_1, \quad & [e_1,e_4]=e_4, \\[2 pt]
[e_1,e_5]=-e_5, \quad & [e_2,e_5]=e_4, \quad  & [e_3,e_4]=e_5  \quad &
\end{array}
\]
and the isotropy algebra is $\h={\rm span}\{h_1=e_3\}$. Taking
\begin{equation}\label{eq:mB1}
\m= {\rm span}\{u_1=e_1,u_2=e_2,u_3=e_4,u_4=e_5\},
\end{equation}
we find that the non-degenerate invariant metrics are
\begin{equation}\label{gB1}
\begin{array}{l} g=a \, \left(2 \theta^1 \circ \theta^3+ 2\theta^2 \circ \theta^4 \right)+b \, \theta^2 \circ \theta^2 +2 c\, \theta^2 \circ \theta^3 +d\, \theta^3 \circ \theta^3, \quad a \neq 0,
\end{array}\end{equation}
where $a$, $b$, $c$ and $d$ are real constants. The metric has neutral signature.

\medskip

The classification of non-reductive homogeneous four-manifolds, given in \cite{FR}, now translates to the following.

\begin{theorem} \label{theo:classnonred}
Let $(M, g)$ be a non-reductive, not locally symmetric homogeneous pseudo-Rieman\-nian four-manifold. If $\g$ is the isometry algebra and $\h$ its isotropy subalgebra, then the Lie algebra pair $(\g,\h)$ is isomorphic to one in the above list: $\mathbf{A1}$, $\mathbf{A2}$, $\mathbf{A3}$, $\mathbf{A4}$, $\mathbf{B1}$ or $\mathbf{B2}$. Conversely, for every Lie algebra pair $(\g,\h)$ in this list there exists a non-reductive homogeneous pseudo-Riemannian four-manifold with isometry algebra $\g$.
\end{theorem}

\begin{remark}
The list of spaces in \cite{FR} also contains a type \textbf{A5} and a type \textbf{B3}. However, it is already mentioned in \cite{FR} that spaces of type \textbf{A5} have constant sectional curvature (in particular, they are locally symmetric) and  spaces of type \textbf{B3} are locally symmetric. Therefore, we omitted them.
\end{remark}

\begin{remark}\label{rem:symmetric cases}
In the list above all invariant metrics are given. However, for certain parameter values $a$, $b$, $c$ and $d$, the invariant metric is locally symmetric. The locally symmetric examples can be easily deduced case by case using the description of the Levi-Civita connection and of the curvature of the invariant metrics reported in Section \ref{sec:LC connection}. We will now list all parameter values for which the invariant pseudo-Riemannian metric is locally symmetric.
\begin{itemize}
\item A space of type \textbf{A1} is locally symmetric if and only if $b=0$. It never has constant sectional curvature.
\item A space of type \textbf{A2} is locally symmetric if and only if $b=0$, in which case it has constant sectional curvature.
\item A space of type \textbf{A3} is locally symmetric if and only if $d+\eta b=0$, in which case it has constant sectional curvature.
\item A space of type \textbf{A4} or \textbf{B2} is locally symmetric if and only if $b=0$, in which case it has constant sectional curvature.
\item A space of type \textbf{B1} is locally symmetric if and only if $c^2-bd=0$ and it has constant sectional curvature if and only if $b=c=d=0$.
\end{itemize}
In what follows we are interested in non-reductive, not locally symmetric spaces and hence we will always assume that none of the above conditions hold.
\end{remark}

	\begin{remark}\label{rem:local lie group}
	It may be observed that for all of the pseudo-Riemannian spaces in Theorem~\ref{theo:classnonred}, except those of type \textbf{A3} with $\eta=1$, the complement $\m$ of $\h$  we choose is a subalgebra of $\g$. This implies that each of them is locally isometric to some Lie group $M_0$ with left invariant metric, which is a discrete quotient of the simply connected Lie group corresponding to $\m$.
	We explicitly note that this local property, just like the local symmetry of an invariant metric, does not contradict in any way the homogeneous space being non-reductive.

		We report in the following Table II the non-vanishing Lie brackets of $\m$ in the different cases, as deduced from the above description, and the simply connected models for the corresponding Lie groups $M_0$ locally isometric to these homogeneous non-reductive spaces 
		(with $\widetilde{SL}(2,\mathbb{R})$, $E(1,1)$ and $H_3$ denoting the universal covering of $SL(2,\mathbb{R})$, the group of rigid motions of the Minkowski plane and the Heisenberg group, respectively).
		
\medskip
\begin{center}
	\begin{tabular}{|l|c|c|c|}
		\hline
		Type $\vphantom{A^{A^A}}$ & Lie brackets of $\m$ &  Model for $M_0$\\
		\hline
		\quad $\textbf{A1}$ $\vphantom{A^{A^{A^A}}}$ & $[u_1,u_2]=2u_2,\; [u_1,u_4]=-2u_4,\; [u_2,u_4]=2u_1$ & $\mathbb{R} \times \widetilde{SL}(2,\mathbb{R})$  \\[2 pt]
		\hline
		\quad  $\textbf{A2}$ $\vphantom{A^{A^A}}$ & 
		$ [u_1,u_4]=(k+1)u_1, \; [u_2,u_4]=k u_2,\;  [u_3,u_4]=(k-1)u_3 $ 
		& $\mathbb{R} \ltimes \mathbb{R}^3$ \\[2 pt]
		\hline
		\quad $\textbf{A3}$ ($\eta=-1$) $\vphantom{A^{A^A}}$ & $[u_1,u_3]=2u_1,\; [u_2,u_3]=u_2,\; [u_2,u_4]=u_2$ & $\mathbb{R} \ltimes E(1,1)$  \\[2 pt]\hline
		\quad  $\textbf{A4}, \textbf{B2}$ $\vphantom{A^{A^A}}$ & $[u_1,u_2]=2u_2,\; [u_1,u_4]=u_4 $ & $\mathbb{R} \ltimes \mathbb{R}^3$   \\[2 pt]\hline
		\quad $\textbf{B1}$ $\vphantom{A^{A^A}}$ & 
		$ [u_1,u_2]=2u_2, \; [u_1,u_3]=u_3,\; [u_1,u_4]=-u_4, \;  [u_2,u_4]=u_3$ 
& $\mathbb{R} \ltimes H_3$   \\[2 pt]\hline
	\end{tabular} 
	
	\nopagebreak $\vphantom{A^{A^A}}${\em Table II: Lie groups locally isometric to non-reductive  \\ \nopagebreak homogeneous pseudo-Riemannian $4$-manifolds}\end{center}

\medskip
For example, for the Lie algebra $\m$ for type $\textbf{A1}$, $u_3$ does not occur in the nonvanishing Lie brackets and the 
three-dimensional Lie algebra $\mathfrak{k}=\mathrm{span}\{u_1,u_2,u_4\}$ satisfies $[\mathfrak{k},\mathfrak{k}]=\mathfrak{k}$, with $\ad _{u_1}$ having two distinct real eigenvalues (besides $0$), so that $\mathfrak{k}= \mathfrak{sl}(2,\R)$, $\m= \mathbb{R}\cdot u_4 \oplus \mathfrak{sl}(2,\R)$ and hence, $M_0$ is modelled on 
$\mathbb{R} \times \widetilde{SL}(2,\mathbb{R})$. For the Lie algebra $\m$ for type $\textbf{A2}$, $u_4$ acts as a derivation on the three-dimensional Lie algebra $\mathfrak{k}=\mathrm{span}\{u_1,u_2,u_3\}$ and $[\mathfrak{k},\mathfrak{k}]=0$, whence $M_0$ is modelled on $\mathbb{R} \ltimes \mathbb{R}^3$. The remaining cases are identified by similar arguments, using the well known classification of three-dimensional real Lie algebras.

	 The properties for a hypersurface to be Codazzi, parallel or totally geodesic are all local. Therefore, for all intended purposes  in all these cases we could work on the Lie group $M_0$. 
		However, in Section~\ref{sec:LC connection} we develop a method which allows us, at least in principle, to investigate all cases without assuming that the ambient space is a Lie group, interpreting vector fields on a homogeneous manifold in terms of corresponding equivariant functions.
	\end{remark}


\subsection{On totally geodesic and parallel hypersurfaces}
Let $F :\Sigma ^n \to M^{n+1}$ be an isometric immersion of pseudo-Riemannian manifolds. We denote both metrics by $\langle \cdot\,,\cdot \rangle$. Let $\xi$ be a normal vector field along the hypersurface, with $\langle \xi,\xi \rangle = \varepsilon \in \{-1,1\}$. Denote by $\nabla ^{\Sigma}$ and $\nabla$ the Levi-Civita connections of $\Sigma$ and $M$ respectively and let $X$, $Y$, $Z$ and $W$ be vector fields on $\Sigma$. We will always identify vector fields on $\Sigma$ with their images under $dF$. The well-known {\em formulas of Gauss and Weingarten},
\begin{align} 
& \nabla_X Y = \nabla ^\Sigma _X Y + h(X,Y) \xi, \label{fG} \\
& \nabla_X \xi = -SX,
\end{align}
define the \emph{second fundamental form} $h$ and the \emph{shape operator} $S$ of the immersion, which are symmetric $(0,2)$- and $(1,1)$-tensor fields on $\Sigma$ respectively, related by $\langle SX,Y \rangle = h(X,Y)$.

A hypersurface is said to be \emph{totally geodesic} if $h=0$ or, equivalently, $S=0$. Remark that a hypersurface is totally geodesic if every geodesic of the hypersurface is also a geodesic of the ambient space. The prime examples are Euclidean subspaces of a Euclidean space.

Consider now the covariant derivatives $\nabla^\Sigma h$ and $\nabla^{\Sigma}S$ of the second fundamental form and the shape operator, given by
\begin{align*} 
& (\nabla^\Sigma h)(X,Y,Z) = X(h(Y,Z)) - h(\nabla^\Sigma _XY,Z) - h(Y,\nabla ^\Sigma _XZ), \\
& (\nabla^\Sigma S)(X,Y) = \nabla_X SY - S \nabla_X Y.
\end{align*} 
The hypersurface is said to be \emph{parallel}, or to have \emph{parallel second fundamental form}, if and only if $\nabla^\Sigma h = 0$ or, equivalently, $\nabla^{\Sigma}S=0$. Obviously, totally geodesic hypersurfaces are special cases of parallel hypersurfaces. We refer to a hypersurface which is parallel, but not totally geodesic, as being \textit{proper parallel}.

Denote by $R^\Sigma$ and $R$ the Riemann-Christoffel curvature tensors of $\Sigma$ and $M$ respectively. The {\em equations of Gauss and Codazzi}, which follow from \eqref{fG} by direct calculation, respectively read
\begin{align}
& \langle R (X,Y)Z,W \rangle = \langle R^\Sigma (X,Y)Z,W \rangle + \varepsilon \left( h(X,Z)h(Y,W) - h(X,W)h(Y,Z) \right), \label{eG} \\
& \varepsilon \langle R (X,Y)Z,\xi \rangle = (\nabla^\Sigma h)(X,Y,Z) - (\nabla^\Sigma h)(Y,X,Z). \label{eC}
\end{align}
Throughout this paper, we use the sign convention $R(X,Y)Z = \nabla_X \nabla_Y Z - \nabla_Y \nabla_X Z -\nabla_{[X,Y]}Z$. 

The hypersurface is said to have a {\em Codazzi second fundamental form} if $\nabla^{\Sigma} h$ is symmetric in its three arguments. Clearly, by equation~\eqref{eC}, this is equivalent to requiring that $R(X,Y)\xi=0$ for all vector fields $X$ and $Y$ on $\Sigma$. In particular, totally geodesic and parallel hypersurfaces have a Codazzi second fundamental form.

\section{Levi-Civita connection and curvature \label{sec:LC connection}}

In this section we first briefly recall how invariant connections on homogeneous spaces are described. {All the theory we recall here, for which we may also refer to  \cite{KobayashiNomizu1963, KobayashiNomizu1969}, needs not assume the homogeneous space being reductive. Next, we explain how vector fields on a homogeneous manifold correspond to certain equivariant functions and describe how an invariant connection acts on these functions. 

Let $M = G/H$ be a homogeneous manifold, let $\g$ be the Lie algebra of $G$ and let $\h$ be the Lie algebra of $H$. For  $X\in \g$, we define a vector field $\overline{X}$ on $M$, generated by the left action of $G$ on $M$ as follows:
\begin{equation*}
\overline{X}_{gH} = \left.\frac{\sd}{d t}\right|_{t=0} \exp(-tX)gH \in T_{gH}M
\end{equation*}
for all $gH \in M$. Note that $\overline{[X,Y]} = [\overline{X},\overline{Y}]$ for all $X,Y\in \g$. Moreover, if $e$ is the identity element of $G$, these vector fields induce a linear map $\g \to T_{eH}M: X \mapsto \overline{X}_{eH}$ with kernel $\h$. Hence, for dimensional reasons, if $\m$ is a complement of $\h$ in $\g$ (not necessarily reductive), this map restricts to a linear isomorphism between vector spaces  $\m$ and $T_{eH}M$. From now on, we will implicitly identify  $\m$ and $T_{eH}M$ under this isomorphism.

Given a $G$-invariant connection $\nabla$ on $M =G/H$, the $\Lambda$-map of the connection is a linear map $\Lambda:\g\to \End(TM)$ defined by
\begin{equation*}
\Lambda(X) = \mathcal{L}_{\overline{X}} - \nabla_{\overline{X}},
\end{equation*}
where $\mathcal{L}_{\overline{X}}$ denotes the Lie derivative with respect to $\overline{X}$.
Let $\lambda:\g\to \End(\m)$ be the evaluation of $\Lambda$ at $eH$, i.e., 
\[
\lambda(X) = \Lambda(X)_{eH}\in \End(T_{eH}M) \cong \End(\m).
\]
The map $\lambda$ is an $H$-equivariant map and satisfies $\lambda(h) = \ad(h)$ for all $h\in \h$. Moreover, $\lambda$ completely determines $\Lambda$ through the identity $$\Lambda(X)_{gH} = dL_g\circ \lambda(\Ad(g)^{-1} X)\circ dL_g^{-1},$$ for every $X\in \g$. The curvature tensor $R$ of an invariant connection is still invariant under the action of $G$. Therefore, evaluating $R$ at $eH$ completely determines $R$. Moreover, $R_{eH}$ can easily be computed by 
\[
R_{eH}(X,Y) = [\lambda(X),\lambda(Y)] - \lambda([X,Y])
\] 
for all $X,Y \in \m$. For the Levi-Civita connection of a $G$-invariant metric, the map $\lambda$ is given by
\begin{equation*}
\langle\lambda(X)Y,Z\rangle = \frac{1}{2}\left(\langle[X,Y]_\m \, ,Z_\m \, \rangle + \langle [Z,X]_\m \, ,Y_\m  \, \rangle + \langle [Z,Y]_\m \, ,X_\m \, \rangle \right),
\end{equation*}
where $X,Y,Z\in \g$, an index $\m$ denotes the projection onto $\m$ along $\h$ and $\langle \cdot, \cdot\rangle$ denotes the inner product on $\m$ induced by the $G$-invariant metric on $M$ via the identification $\m \cong T_{eH}M$. Conversely, given an $H$-equivariant linear map $\lambda:\g\to\End(\m)$, such that $\lambda(h) = \ad(h)$, then $\lambda$ defines a homogeneous connection on $M=G/H$.

We now explain how we can identify vector fields on a homogeneous space $M=G/H$ with functions from $G$ to $\m$. Given a vector field $X\in \Gamma(TM)$,  taking into account the above identification $\m \cong T_{eH}M$, we define $\hat{X}$ by
\[
\begin{array}{rcl}
\hat X: \; G & \to & \m \\[4pt]
g & \mapsto  & \hat{X}(g) = dL_g^{-1} X_{gH}.
\end{array}
\]
Note that the function $\hat{X}$ is $H$-equivariant, i.e., $\hat{X}(gh) = (\Ad(h)^{-1} X)_\m \, $. Conversely, given an $H$-equivariant function $\hat{X}:G \to \m$, the above formula defines a vector field on $M$. From now on we will identify vector fields on $M$ with these $H$-equivariant functions and drop the hat from the notation. In particular, $X(g)$ will denote the value of the function corresponding to the vector field~$X$ at $g$ and $X_{gH}$ denotes the value of the vector field at $gH$. 

Let $\nabla$ be an invariant connection on $M=G/H$ with corresponding map $\lambda:\g\to \End(\m)$. Let $\underline{\m}=G\times \m$ be the trivial vector bundle and denote by $\{u_1,\ldots,u_n\}$ a basis for $\m$. The \emph{absolute differential} corresponding to $\nabla$, is the operator defined by  
\[
\begin{array}{rrcl}
D: & \Gamma(TG)\times \Gamma(\underline{\m}) &\to & \Gamma(\underline{\m}) \\[6pt] 
& \left( X, \, Y=\sum_{i=1}^n y_i u_i \right) & \mapsto &
D_X Y = \sum_{i=1}^n X(y_i) u_i + (\lambda\circ X)Y.
\end{array}
\]
If $X$ and $Y$ are vector fields on $M$, we can interpret them as $H$-equivariant maps from $G$ to $\m \subseteq \g$ as explained above and hence as elements of $\Gamma(\underline{\m})$ or $\Gamma(TG)$. Thus it makes sense to compute $D_XY$ and we find $D_X Y = \nabla_X Y$. The last equation can  be interpreted either as an equation of vector fields or as an equation of functions. 

The advantage of working with the absolute derivative is that we can do all computations on functions $G\to \m$ which need not be equivariant. For example, if $F:\Sigma\to M$ is a hypersurface and $\xi$ is a unit normal vector field along the hypersurface, identified with the corresponding function $G\to \m$, then the shape operator can simply be computed by $S(X) = -D_{X} \xi(g)$, for any function $X:G\to \m$ such that $X(g)$ is perpendicular to $\xi(g)$ for all $g\in G$.
The shape operator is parallel if and only if the tensorial condition
\[
\nabla_X SY - S\nabla_X Y =0
\]
holds for all $X,Y\in \Gamma(T\Sigma)$. Equivalently, we can check this condition with the absolute derivative and we find that $S$ is parallel if and only if
\[
D_V SW - SD_V W =0
\]
holds for all $V,W:G\to \m$ such that $V(g),W(g)\perp \xi(g)$.
%
%
%

We now describe the Levi-Civita connection and the Riemann-Christoffel curvature tensor of all spaces appearing in Theorem \ref{theo:classnonred},  with respect to the basis $\{u_1,u_2,u_3,u_4\}$ given in Section \ref{sec:preliminaries} for each of them.

\smallskip\noindent
{\bf Type A1:}
{\small
\begin{equation}\label{lambdaA1}
\begin{array}{ll}
\lambda(u_1) = \begin{pmatrix}
0 & 0 & 0 & 0\\
0 & 1 & 0 & 0\\
0 & 0 & 0 & 0\\
0 & -\frac{b}{c} & -\frac{c}{a} & -1
\end{pmatrix}, & \lambda(u_2) = \begin{pmatrix}
0 & -\frac{8db}{a(a-4d)} & \frac{c}{a} & 1\\
-1 & 0 & \frac{1}{2} & 0\\
0 & -\frac{4b}{a-4d} & 0 & 0\\
-\frac{b}{a} & \frac{4cb}{a(a-4d)} & -\frac{b}{2a} & 0\\
\end{pmatrix} \vspace{2mm}
\\

\lambda(u_3) = \begin{pmatrix}
0 & \frac{c}{a} & 0 & 0\\
0 & \frac{1}{2} & 0 & 0\\
0 & 0 & 0 & 0\\
-\frac{c}{a} & -\frac{b}{2a} & 0 & -\frac{1}{2}
\end{pmatrix}, & \lambda(u_4) =\begin{pmatrix}
0 & -1 & 0 & 0\\
0 & 0 & 0 & 0\\
0 & 0 & 0 & 0\\
1 & 0 & -\frac{1}{2} & 0
\end{pmatrix}.
\end{array}
\end{equation}

\begin{equation}\label{RA1}
\begin{array}{ll}
R(u_1,u_2) =\left( \begin {array}{cccc} 0& \frac {b (a+20 d)}{\alpha\, \left(a -4d \right)}&-{\frac {c}{a}}&-1\\1&0&-\frac 12 &0\\0&\,{\frac {12b}{a-4d}}&0&0\\ \noalign{\smallskip}{\frac {4b}{a}}&-\,{\frac {12bc}{a \left(a -4d \right) }}&{\frac {b}{a}}&0\end {array} \right), &R(u_1,u_3)=\left( \begin {array}{cccc} 0&-{\frac {c}{a}}&0&0\\0&0&0&0\\0&0&0&0\\{\frac {c}{a}}&0&-\,{\frac {c}{2a}}&0\end {array} \right), 
\vspace{2mm} \\ 
R(u_1,u_4) =\left( \begin {array}{cccc} 0& -1& 0 &0 \\0&0&0&0\\ 0& 0 &0&0\\ 1 &0 & -\frac 12 &0\end {array} \right), &R(u_2,u_3)=\left( \begin {array}{cccc} 0&-\frac {
b(a +4 d)}{2a  \left(a -4d \right) }&-\,{\frac {c}{2a}}&-\frac 12\\ \noalign{\smallskip}\frac 12 &-{\frac {c}{a}}&-\frac 14&0\\\noalign{\smallskip}0&-{\frac {2b}{a -4d}}&0&0\\\noalign{\smallskip}-{\frac {b}{a}}&\frac {bc (3a-4d)}{a^{2} (a -4d) }  &{\frac {c^{2}}{a^{2}}}&{\frac {c}{a}}\end {array} \right), 
\vspace{2mm}\\ 
R(u_2,u_4)=\left( \begin {array}{cccc} 0&0&0&0\\ 0&-1&0&0\\ 0&0&0&0\\ 0&{\frac {b}{a}}&{\frac {c}{a}}&1\end {array} \right), &R(u_3,u_4)=\left( \begin {array}{cccc} 0&\frac 12 &0&0\\ 0&0&0&0\\ 0&0&0&0\\ -\frac 12& 0 & {\frac {1}{4}}&0\end {array} \right)  .
\end{array}
\end{equation}
}

\smallskip\noindent
{\bf Type A2:}
{\small \begin{equation}\label{lambdaA2}
\begin{array}{ll}
\lambda(u_1)=\begin{pmatrix}
0 & 0 & \frac{ \kappa c}{d} & \kappa \\ 0 & 0 & 0 & 0 \\ 0 & 0 & 0 & 0 \\ 0 & 0 & \frac{ \kappa\,a}{d}  & 0\end{pmatrix}, &\lambda(u_2)=\left( \begin {array}{cccc} 0&-{\frac {\kappa c}{d}}&0&0\\ 0&0&0&\kappa \\ 0&0&0&0\\0&-{\frac {\kappa \, a}{d}}&0&0\end {array} \right), 
\vspace{2mm}\\
\lambda(u_3)=\left( \begin {array}{cccc} {\frac {\kappa\,c}{d}}&0&-{\frac { \left( \kappa-1 \right) bc}{ad}}&-{\frac {\kappa\,{c}^{2}-bd}{ad}}\\0&0&0&0\\0&0&0&\kappa\\{\frac {\kappa\,a}{d}}&0&-{\frac { \left( \kappa-1 \right) b}{d}}&-{\frac {\kappa\,c}{d}}\end {array} \right), &\lambda(u_4)=\left( \begin {array}{cccc} -1&0&-{\frac {\kappa\,{c}^{2}-bd}{ad}}&-{\frac { \left( \kappa-1 \right) c}{a}}\\0&0&0&0\\0&0&1&0\\0&0&-{\frac {\kappa\,c}{d}}&0\end {array} \right).
\end{array}
\end{equation}
\begin{equation}\label{RA2}
\begin{array}{ll}
R(u_1,u_2) =\left( \begin {array}{cccc} 0&-{\frac {{\kappa}^{2}a}{d}}&0&0\\0&0&-{\frac {{\kappa}^{2}a}{d}}&0\\0&0&0&0\\0&0&0&0\end {array} \right), &R(u_1,u_3)=\left( \begin {array}{cccc} {\frac {{\kappa}^{2}a}{d}}&0&-{\frac {\kappa ^2 b}{d}}&-{\frac {{\kappa}^{2}c}{d}}\\0&0&0&0\\0&0&-{\frac {{\kappa}^{2}a}{d}}&0\\0&0&0&0\end {array} \right), \vspace{2mm}\\
R(u_1,u_4)=\left( \begin {array}{cccc} 0&0&-{\frac { \kappa ^2 \,c}{d}}& -\kappa ^2 \\0&0&0&0\\ 0&0&0&0\\ 0&0&-{\frac { \kappa ^2 \,a}{d}}&0\end {array} \right), &R(u_2,u_3)=\left( \begin {array}{cccc} 0&{\frac {{\kappa} b}{d}}&0&0\\{\frac {{\kappa}^{2}a}{d}}&0&-{\frac {\kappa\, \left( \kappa-1 \right) b}{d}}&-{\frac {{\kappa}^{2}c}{d}}\\0&{\frac {{\kappa}^{2}a}{d}}&0&0\\0&0&0&0\end {array} \right), \vspace{2mm}\\
R(u_2,u_4)=\left( \begin {array}{cccc} 0&0&0&0\\0&0&-{\frac {{\kappa}^{2}c}{d}}&-{\kappa}^{2}\\0&0&0&0\\0&{\frac {{\kappa}^{2}a}{d}}&0&0\end {array} \right), &R(u_3,u_4)=\left( \begin {array}{cccc} 0&0&-{\frac { 2\left( \kappa-1 \right) bc}{ad}} &-{\frac { 2\left( \kappa -1\right)  \, b  }{a}}\\0&0&0&0\\0&0&-{\frac {{\kappa}^{2}c}{d}}& -\kappa^2 \\-{\frac {\kappa^2 a}{d}}&0&{\frac { \left( \kappa ^2 -2\kappa +2\right)  b}{d}}&{\frac { \kappa ^2 \,c}{d}} \end {array}\right).
\end{array}
\end{equation}
}

\smallskip\noindent
{\bf Type A3:}
{\small \begin{equation}\label{lambdaA3}
\begin{array}{ll}
\lambda(u_1)=\left( \begin {array}{cccc} 0&0&1&{\frac {c}{b}}\\0&0&0&0\\0&0&0&-{\frac {a}{b}}\\0&0&0&0\end {array} \right) , & 
\lambda(u_2)=\left(\begin {array}{cccc}
 0 & {\frac {c}{b} -1} & 0 & 0\\
 0 & 0 & 1 & 1\\
 0 & -{\frac {a}{b}} & 0 & 0\\
 0 & 0 & 0 & 0
 \end {array} \right), 
\vspace{2mm}\\

\lambda(u_3)=\left( \begin {array}{cccc} -1&0&0&{\frac {{c}^{2}-bd}{ab}}\\0&0&0&0\\0&0&0&-{\frac {c}{b}}\\0&0&0&1\end {array} \right), &\lambda(u_4)=\left( \begin {array}{cccc} {\frac {c}{b}}&0&{\frac {{c}^{2}-bd}{ab}}&0\\0&0&0&0\\-{\frac {a}{b}}&0&-{\frac {c}{b}}&0\\0&0&1&0\end {array} \right). 
\end{array}
\end{equation}
\begin{equation}\label{RA3}
\begin{array}{ll}
R(u_1,u_2) =\left( \begin {array}{cccc} 0&-{\frac {a}{b}}&0&0\\ 0&0&0&{\frac {a}{b}}\\ 0&0&0&0\\ 0&0&0&0\end {array} \right), &R(u_1,u_3)=\left( \begin {array}{cccc} 0&0&-1&-{\frac {c}{b}}\\ 0&0&0&0\\ 0&0&0&{\frac {a}{b}}\\ 0&0&0&0\end {array} \right), \vspace{2mm}\\

R(u_1,u_4)=\left( \begin {array}{cccc} -{\frac {a}{b}}&0&-{\frac {c}{b}}&-{\frac {d}{b}}\\0&0&0&0\\ 0&0&0&0\\ 0&0&0&{\frac {a}{b}}\end {array} \right), &R(u_2,u_3)=\left( \begin {array}{cccc} 0&0&0&0\\ 0&0&-1&-{\frac {c}{b}}\\ 0&{\frac {a}{b}}&0&0\\ 0&0&0&0\end {array} \right), 
\vspace{2mm}\\

R(u_2,u_4)=\left( \begin {array}{cccc} 0&-\eta-{\frac {{d}}{b}}&0&0\\ -{\frac {a}{b}}&0&-{\frac {c}{b}}&\eta \\ 0&0&0&0\\ 0&{\frac {a}{b}}&0&0\end {array} \right), &R(u_3,u_4)=\left( \begin {array}{cccc} 0&0&0&0\\0&0&0&0\\ -{\frac {a}{b}}&0&-{\frac {c}{b}}&-{\frac {d}{b}}\\ 0&0&1&{\frac {c}{b}}\end {array} \right).
\end{array}
\end{equation}
}

\smallskip\noindent
{\bf Types A4 and B2:}
{\small \begin{equation}\label{lambdaA4}
\begin{array}{ll}
\lambda(u_1) = \begin{pmatrix}
0 & 0 & 0 & 0\\
0 & 1 & 0 & 0\\
0 & -\frac{b}{a} & -1 & 0\\
0 & 0 & 0 & 0
\end{pmatrix}, &
\lambda(u_2) = \begin{pmatrix}
0 & \frac{2b}{a}\eta & \eta & 0\\
-1 & 0 & 0 & 0\\
-\frac{b}{a} & 0 & 0 & 0\\
0 & 0 & 0 & 0
\end{pmatrix} \vspace{2mm}
\\
\lambda(u_3) = \begin{pmatrix}
0 & \eta & 0 & 0\\
0 & 0 & 0 & 0\\
-1 & 0 & 0 & 0\\
0 & 0 & 0 & 0
\end{pmatrix}, &
\lambda(u_4) = \begin{pmatrix}
0 & 0 & 0 & \frac{\eta}{2}\\
0 & 0 & 0 & 0\\
0 & 0 & 0 & 0\\
-1 & 0 & 0 & 0
\end{pmatrix}
\end{array}
\end{equation}
\begin{equation}\label{RA4}
\begin{array}{ll}
R(u_1,u_2) = \begin{pmatrix}
0 & -\frac{5b}{a}\eta & -\eta & 0\\
1 & 0 & 0 & 0\\
\frac{4b}{a} & 0 & 0 & 0\\
0 & 0 & 0 & 0
\end{pmatrix}, &
R(u_1,u_3) = \begin{pmatrix}
0 & -\eta & 0 & 0\\
0 & 0 & 0 & 0\\
1 & 0 & 0 & 0\\
0 & 0 & 0 & 0
\end{pmatrix}, \vspace{2mm} \\ 
R(u_1,u_4) = \begin{pmatrix}
0 & 0 & 0 & -\frac{\eta}{2}\\
0 & 0 & 0 & 0\\
0 & 0 & 0 & 0\\
1 & 0 & 0 & 0
\end{pmatrix}, &
R(u_2,u_3) = \begin{pmatrix}
0 & 0 & 0 & 0\\
0 & -\eta & 0 & 0\\
0 & \frac{b}{a}\eta & \eta & 0\\
0 & 0 & 0 & 0
\end{pmatrix}, \vspace{2mm} \\ 
R(u_2,u_4) = \begin{pmatrix}
0 & 0 & 0 & 0\\
0 & 0 & 0 & -\frac{\eta}{2}\\
0 & 0 & 0 & -\frac{b}{2a}\eta \\
0 & \frac{2b}{a}\eta & \eta & 0
\end{pmatrix}, &
R(u_3,u_4) = \begin{pmatrix}
0 & 0 & 0 & 0\\
0 & 0 & 0 & 0\\
0 & 0 & 0 & -\frac{\eta}{2}\\
0 & \eta & 0 & 0
\end{pmatrix}.
\end{array}
\end{equation}
}

\smallskip\noindent
{\bf Type B1:}
{\small \begin{equation}\label{lambdaB1}
\begin{array}{ll}
\lambda(u_1)=\left(  \begin {array}{cccc} -1&-{\frac {c}{2a}}&-{\frac {d}{a}}&0\\0&1&0&0\\0&0&1&0\\0&-{\frac {b}{a}}&-{\frac {3c}{2a}}&-1\end {array} \right) , &\lambda(u_2)=\left( \begin {array}{cccc} -{\frac {c}{2a}}&{\frac {{c}^{2}-2bd}{{a}^{2}}}&-{\frac {cd}{{a}^{2}}}&-{\frac {d}{2a}}\\\noalign{\smallskip}-1&-{\frac {c}{a}}&-{\frac {d}{2a}}&0\\\noalign{\smallskip}0&{\frac {2b}{a}}&{\frac {3c}{2a}}&1\\\noalign{\smallskip}-{\frac {b}{a}}&-{\frac {bc}{{a}^{2}}}&{\frac {bd-3{c}^{2}}{2{a}^{2}}}&0\end {array} \right), 
\vspace{2mm} \\
\lambda(u_3)=\left( \begin {array}{cccc} -{\frac {d}{a}}&-{\frac {cd}{{a}^{2}}}&-{\frac {{d}^{2}}{{a}^{2}}}&0\\\noalign{\medskip}0&-{\frac {d}{2a}}&0&0\\\noalign{\medskip}0&{\frac {3c}{2a}}&{\frac {d}{a}}&0\\\noalign{\medskip}-{\frac {3c}{2a}}&{\frac {bd-3\,{c}^{2}}{2{a}^{2}}}&-{\frac {cd}{{a}^{2}}}&{\frac {d}{2a}}\end {array} \right), &\lambda(u_4)=\left( \begin {array}{cccc} 0&-{\frac {d}{2a}}&0&0\\ 0&0&0&0\\0&0&0&0\\0&0&{\frac {d}{2a}}&0\end {array} \right). 
\end{array}\end{equation}
\begin{equation}\label{RB1}
\begin{array}{ll}
R(u_1,u_2) =\left( \begin {array}{cccc} {\frac {3c}{2a}}&{\frac {22\,bd-15\,{c}^{2}}{4{a}^{2}}} &{\frac {3cd}{2{a}^{2}}}&0\\\noalign{\smallskip}0&{\frac {3c}{2a}}&0&0\\\noalign{\smallskip}0&-{\frac {3b}{a}}&-{\frac {3c}{2a}}&0\\\noalign{\smallskip}{\frac {3b}{a}}&{\frac {3bc}{2{a}^{2}}}&{\frac {5 \left(3{c}^{2}-2\,bd \right)}{4{a}^{2}}}&-{\frac {3c}{2a}}\end {array} \right), &R(u_1,u_3)=\left(\begin {array}{cccc} {\frac {d}{a}}&{\frac {5cd}{4{a}^{2}}}&{\frac {{d}^{2}}{{a}^{2}}}&0\\\noalign{\smallskip}0&{\frac {d}{2a}}&0&0\\\noalign{\smallskip}0&-{\frac {3c}{2a}}&-{\frac {d}{a}}&0\\\noalign{\smallskip}{\frac {3c}{2a}}&{\frac {3{c}^{2}-bd}{2{a}^{2}}}&{\frac {3cd}{4{a}^{2}}}&-{\frac {d}{2a}}\end {array} \right), \vspace{2mm}\\
R(u_2,u_3)=\left( \begin {array}{cccc} -{\frac {cd}{4{a}^{2}}}&\frac {3d\left( bd-c^2 \right) }{4{a}^{3}}&0&{\frac {{d}^{2}}{4{a}^{2}}}\\\noalign{\smallskip}{\frac {d}{2a}}&{\frac {cd}{4{a}^{2}}}&{\frac {{d}^{2}}{4{a}^{2}}}&0\\\noalign{\smallskip}0&{\frac {9c^2 -10bd}{4{a}^{2}}}&-{\frac {cd}{4{a}^{2}}}&-{\frac {d}{2a}}\\\noalign{\smallskip} {\frac {8bd -9 c^2}{4{a}^{2}}}&{\frac {9c \left(bd-c^2\right)}{4{a}^{3}}}&{\frac {3d\left( b{d} -c^{2}\right)}{2{a}^{3}}} & {\frac {cd}{4{a}^{2}}}\end {array} \right), & R(u_1,u_4)=\left( \begin {array}{cccc} 0&{\frac {d}{2a}}&0&0\\ 0&0&0&0\\0&0&0&0\\0&0&-{\frac {d}{2a}}&0\end {array} \right),  
\vspace{2mm}\\
R(u_2,u_4)=\left( \begin {array}{cccc} {\frac {d}{2a}}&{\frac {3cd}{4{a}^{2}}}&{\frac {{d}^{2}}{2{a}^{2}}}&0\\ \noalign{\smallskip}0&{\frac {d}{a}}&0&0\\ \noalign{\smallskip}0&-{\frac {3c}{2a}}&-{\frac {d}{2a}}&0\\\noalign{\smallskip} {\frac {3c}{2a}}&{\frac {3\,{c}^{2}-2bd}{2{a}^{2}}}& \frac{cd}{4a^2}&-{\frac {d}{a}}\end {array} \right), & R(u_3,u_4)=\left( \begin {array}{cccc} 0&{\frac {{d}^{2}}{4{a}^{2}}}&0&0\\0&0&0&0\\0&0&0&0\\0&0&-{\frac {{d}^{2}}{4{a}^{2}}}&0\end {array} \right). 
\end{array}
\end{equation}
}


\section{Hypersurfaces of homogeneous spaces of type \textbf{A1}}
\setcounter{equation}{0}

Let $(M,g)$ be a homogeneous space of type \textbf{A1}. We first determine some necessary algebraic conditions on the components of a unit normal vector field $\xi$ along a hypersurface of $(M,g)$ in order for the hypersurface to have a Codazzi second fundamental form. In particular, we prove the following.

\begin{lemma} \label{necA1}
Let $(M,g)$ be a homogeneous pseudo-Riemannian four-manifold of type {\rm\bf A1} with metric~\eqref{gA1}. If there exists a non-degenerate hypersurface $F : \Sigma \to M$ with Codazzi second fundamental form, then $b=0$.
\end{lemma}

\begin{proof}
Denote by $\{u_1,u_2,u_3,u_4\}$ the basis of $\m$ given in \eqref{eq:mA1} and let $\alpha, \beta, \gamma, \delta: M \to \R$ be functions such that the vector field $\xi=\alpha u_1+\beta u_2+\gamma u_3+\delta u_4$ is normal to $\Sigma$ and satisfies $\langle \xi,\xi \rangle =\varepsilon=\pm 1$. For an arbitrary function $X = x_1u_1+x_2u_2+x_3u_3+x_4u_4$, we have from \eqref{gA1}
\[
\langle X, \xi\rangle = a\left(\alpha - \frac{\gamma}{2} \right) x_1+  \left( a\delta+b\beta+c\gamma \right) x_2+ \left( c\beta+d\gamma-\frac{a\alpha}{2} \right) x_3 + a\beta\,x_4,
\]
so that
\[
\langle \xi, \xi\rangle  = a\alpha^2-a \alpha\gamma+2a\beta\delta+b\beta^2+2c\beta\gamma+d\gamma^2 = \varepsilon.
\]

First suppose $\beta \neq 0$ and let $X = 2\beta\, u_1 + (\gamma-2\alpha)\,u_4$ and $Y= 2a\beta\, u_3 + (a\alpha - 2c\beta -2d\gamma)\, u_4$. Then $X$ and $Y$ are perpendicular to $\xi$ and $R(X,Y)\xi = 0$ implies $\gamma = 0$. Now let $Z=a\beta\,u_2 - (a\delta+b\beta)\, u_4$ then $Z$ is perpendicular to $\xi$ and $R(Y,Z)\xi =0$ implies $b=0$.

Now suppose $\beta  =0 $. Let $X= u_4$ and $Y=(2d\gamma-a\alpha)\, u_1 + a(\gamma-2\alpha)\, u_3$ be two functions perpendicular to $\xi$. We have $R(X,Y)\xi = 0$ if and only if $\gamma = 0$ or $\gamma-2\alpha  = 0$. First consider the subcase  $\gamma = 0$. Then $Z= \delta\,u_1 - \alpha\,u_2$ is perpendicular to $\xi$ and $R(Y,Z)\xi = 0$ implies $b=0$. For the second subcase, $\gamma-2\alpha=0$, consider $Z = u_1$ and $W=\left(4d\alpha - a\alpha\right)\,u_2 - 2(a\delta +2c\alpha)\, u_3$. Then $Z$ and $W$ are perpendicular to $\xi$ and $R(Z,W)\xi = 0$ again implies $b=0$.
\end{proof}

By Remark~\ref{rem:symmetric cases}, a homogeneous space of type \textbf{A1} with $b=0$ is locally symmetric. Hence, we proved the following.

\begin{theorem}
A non-reductive, not locally symmetric homogeneous pseudo-Riemannian four-man\-if\-old of type {\rm\bf A1} does not allow hypersurfaces with Codazzi second fundamental form. In particular, it does not allow parallel and totally geodesic hypersurfaces.
\end{theorem}
 



\section{Hypersurfaces of homogeneous spaces of type \textbf{A2}}
\setcounter{equation}{0}

In this case, the transitive Lie algebra $\g$ depends on a real parameter $\kappa$, which increases the computational difficulties of the classification problem we are dealing with. For this reason, this is the only case where we shall explicitly make use of the fact that 
$\m\subseteq \g$ is a subalgebra, and all our computations will be done on a Lie group $M_0$ with Lie algebra $\m$, see Remark~\ref{rem:local lie group}. 
The identity component of the automorphism group of $\m$ with respect to the basis $\{u_1,u_2,u_3,u_4\}$ is given by
\[
\left\{ \left. \phi(x_1,x_2,x_3,x_4,x_5,x_6)= \begin{pmatrix}
x_1 & 0 & 0 & x_6\\
0 & x_2 & 0 & x_5\\
0 & 0 & x_3 & x_4\\
0 & 0 & 0 & 1
\end{pmatrix} \ \right| \ x_1,\dots,x_6\in \R \mbox{ and } x_1x_2x_3 \neq 0 \right\}.
\]
By Remark~\ref{rem:symmetric cases}, we only need to consider the case where  $b\neq 0$. Then, the metric \eqref{gA2} pulled back under $\phi\left(\frac{\sqrt{b}}{a},\frac{1}{\sqrt{a}}, \frac{1}{\sqrt{b}}, 0, 0,\frac{c}{a}\right)$ is equal to
\begin{equation}\label{metA2}
g_0= -2\theta^1 \circ \theta^3+\theta^2 \circ \theta^2 + \theta^3 \circ \theta^3 +d\, \theta^4 \circ \theta^4, \quad d \neq 0.
\end{equation}
We can integrate the Lie algebra automorphism $\phi$ to a Lie group automorphism $\Phi$ of $M_0$. This map $\Phi:M_0\to M_0$ constitutes an isometry between the left invariant metric  $g_0$ and the left invariant metric given by \eqref{gA2}.
Therefore, without loss of generality we can restrict to the case with $a=b=1$ and $c=0$. This will simplify some computations.

\begin{lemma}\label{necA2}
Let $(M,g)$ be a non-reductive homogeneous pseudo-Riemannian four-manifold of type {\rm\bf A2} with metric \eqref{gA2} and denote by $\{u_1,u_2,u_3,u_4\}$ the basis of $\m$ given in \eqref{eq:mA2}. Let $F : \Sigma \to M$ be a hypersurface and $\alpha, \beta, \gamma, \delta : M \to \R$ functions such that the vector field $\xi=\alpha u_1+\beta u_2+\gamma u_3+\delta u_4$ is normal to $\Sigma$ and satisfies $\langle \xi,\xi \rangle =\varepsilon=\pm 1$. Then $\Sigma$ has a Codazzi second fundamental form if and only if one of the following sets of conditions holds: 
\begin{itemize}
\item[(i)] $\beta=\gamma=0$ and $d\delta^2=\varepsilon$;
\item[(ii)] $\gamma=\delta=0$ and $\beta^2=\varepsilon=1$;
\item[(iii)] $\kappa=2$, $\gamma=0$, $\beta \neq 0$, $\delta\neq 0$ and $\beta^2+d\delta^2=\varepsilon$. 
\end{itemize}
\end{lemma}

\begin{proof} It follows from \eqref{gA2} that, for an arbitrary function $X=x_1u_1+x_2u_2+x_3u_3+x_4u_4$, 
\[
\langle X,\xi\rangle=-\gamma x_1+ \beta x_2-( \alpha-\gamma)x_3+d\delta x_4,
\]
in particular,
\begin{equation}\label{nxiA2}
\langle \xi,\xi\rangle=\beta^2 -2\alpha \gamma +\gamma ^2 +d\delta^2 = \varepsilon.
\end{equation}
First suppose that $\gamma \neq 0$ and consider the functions $X =\beta u_1 + \gamma u_2$, $Y =d\delta u_1 + \gamma u_4$ and  $Z=(\alpha - \gamma)u_1 - \gamma u_3$. They are perpendicular to $\xi$ and the equations $R(X,Z)\xi = 0$ and $R(Y,Z)\xi =0$ have no solutions. 

Now suppose $\gamma = 0$. If $\beta = 0$ we are in case (i), so we may assume that $\beta \neq 0$. Then $X= \alpha u_2 + \beta u_3$ and $Y= d\delta u_2 - \beta u_4$ are perpendicular to $\xi$ and $R(X,Y)\xi = 0$ is equivalent to $\delta = 0$ or $\kappa = 2$. 

For the converse, remark that the following functions span an orthonormal basis of $\xi(g)$ for all $g\in G$:
\begin{align}
& \mbox{Case (i)}: & V_1=u_1, &&& V_2=u_2, && V_3=d\delta u_3+\alpha u_4, \label{basisA2i} \\
& \mbox{Case (ii)}: & V_1=u_1, &&& V_2=\alpha u_2+\beta u_3, && V_3=u_4, \label{basisA2ii} \\
& \mbox{Case (iii)}: & V_1=u_1, &&& V_2=d \delta u_2-\beta u_4, && V_3=\alpha u_2+\beta u_3. \label{basisA2iii}
\end{align}
It is now sufficient to check that, in each of the three cases, $R(V_i,V_j)\xi =0$ for all $i,j \in \{1,2,3\}$.
\end{proof}

In Lemma~\ref{necA2}, we did not use the fact that $\mathrm{span}\{\xi\}^{\perp}=\mathrm{span}\{V_1,V_2,V_3\}$ is integrable. Imposing integrability leads to differential conditions on the component functions of $\xi$, in addition to the algebraic conditions already obtained in Lemma~\ref{necA2}. 

\begin{theorem}\label{codA2}
In the notations of Lemma~{\rm\ref{necA2}}, the hypersurface $\Sigma$ has Codazzi second fundamental form if and only if one of the following sets of conditions holds: 
\begin{itemize}
\item[(i)] $\beta=\gamma=0$, $d\delta^2 = \varepsilon$ and 
\[
V_1(\alpha) = 0, \quad V_2(\alpha)=0,
\]
where $V_1,V_2$ and $V_3$ are the functions described in \eqref{basisA2i};  
\vspace{1mm}\item[(ii)] $\gamma=\delta=0$, $\beta^2 = \varepsilon =1$ and 
\[
V_1(\alpha)=0, \quad V_3(\alpha)-\alpha=0,
\]
where $V_1,V_2$ and $V_3$ are the functions described in \eqref{basisA2ii};
\vspace{1mm}\item[(iii)] $\kappa = 2$, $\gamma =0$, $\beta\neq 0$, $\delta\neq 0$, $\beta^2 + d\delta^2 = \varepsilon$ and 
\[
V_1(\alpha) = 0, \quad V_1(\beta) = 0, \quad \beta V_2(\alpha) - \alpha V_2(\beta) + \frac{\varepsilon}{\delta}V_3(\beta) + \alpha\beta^2 =0, 
\]
where $V_1,V_2$ and $V_3$ are the functions described in \eqref{basisA2iii}.
\end{itemize}
\end{theorem}
\begin{proof}
We treat the three cases listed in Lemma \ref{necA2} separately. Remark that the integrability of the distribtution $\mathrm{span}\{\xi\}^{\perp}=\mathrm{span}\{V_1,V_2,V_3\}$ is equivalent to the symmetry of the shape operator $S: \mathrm{span}\{V_1,V_2,V_3\} \to \mathrm{span}\{V_1,V_2,V_3\} : V \mapsto -\nabla_V\xi$.

\smallskip
{\em Case} (i). Remark that $\delta$ is a non-zero constant. Using \eqref{lambdaA2} and \eqref{basisA2i}, we express $SV_j=-\nabla_{V_j}\xi$ in the basis $\{V_1,V_2,V_3\}$ for every $j=1,2,3$. We find
\begin{equation}\label{sA2i}
\left\{
\begin{array}{l}
SV_1 =   -(\kappa\delta + V_1(\alpha))V_1, \\[4 pt]
SV_2 = - V_2(\alpha) V_1   -  \kappa \delta V_2, \\[4 pt]
SV_3 =  \left(-V_3(\alpha) + \alpha^2-\varepsilon \right) V_1 -\kappa\delta V_3.
\end{array}
\right.
\end{equation}
It follows that $S$ is symmetric with respect to the metric \eqref{metA2} if and only if $V_1(\alpha)=V_2(\alpha)=0$. 

\smallskip 
{\em Case} (ii). In this case $\beta \neq 0$ is a constant and the shape operator $S$ with respect to the basis \eqref{basisA2ii} is given by
\begin{equation}\label{sA2ii}
\left\{
\begin{array}{l}
SV_1 =  - V_1(\alpha) V_1, \\[2 pt]
SV_2 =  -V_2(\alpha) V_1, \\[2 pt]
SV_3 =  \left(\alpha - V_3(\alpha) \right) V_1,
\end{array}
\right.
\end{equation}
which is symmetric with respect to \eqref{metA2} if and only if $V_1(\alpha)=V_3(\alpha)-\alpha=0$. 

\smallskip
{\em Case} (iii). Using $\beta^2 + d \delta^2 = \varepsilon$ to eliminate derivatives of $\delta$,  the shape operator with respect to \eqref{basisA2iii} is given by 
\begin{equation}\label{sA2iii}
\left\{
\begin{array}{l}
SV_1 =  -(V_1(\alpha)+2\delta) V_1 -\frac{1}{d\delta}V_1(\beta) V_2, \\[2 pt]
SV_2 =  -(V_2(\alpha)+\alpha\beta) V_1  - \left(2\delta+\frac{V_2(\beta)}{d\delta} \right)V_2, \\[2 pt]
SV_3 =  -\left(V_3(\alpha) +\beta \delta \right) V_1 - \frac{V_3(\beta)}{d\delta} V_2 - 2\delta V_3,
\end{array}
\right.
\end{equation}
which is symmetric if and only if $V_1(\alpha) =V_1(\beta) =0$ and $(V_2(\alpha)+\alpha\beta)\beta - V_2(\beta)\alpha + \frac{\varepsilon}{\delta}V_3(\beta) =0$.
\end{proof}

We can now classify parallel and totally geodesic hypersurfaces of a non-reductive homogeneous pseudo-Riemannian space of type~\textbf{A2}. 

\begin{theorem}\label{parA2}
In the notations of Lemma~{\rm\ref{necA2}}, the hypersurface $\Sigma$ is parallel if and only if one of the following sets of conditions holds.
\begin{itemize}
\item[(1)]  $\beta=\gamma=0$, $d\delta^2 = \varepsilon$ and
\[
V_1(\alpha)=V_2(\alpha)=0, \quad V_3(\varphi) = 2(\kappa+1)\alpha \varphi
\]
where $V_1,V_2$ and $V_3$ are the functions given in \eqref{basisA2i} and we put $\varphi = \alpha^2- \varepsilon -V_3(\alpha)$. In particular, $\Sigma$ is totally geodesic if and only if $V_1(\alpha) = V_2(\alpha) =0$ and $\kappa=\varphi = 0$.
\vspace{1mm} 
\item[(2)] $\gamma=\delta=0$, $\beta^2=\varepsilon=1$ and either
\[
V_1(\alpha)=V_3(\alpha)-\alpha=0,\quad V_2(\alpha) =0,
\]
or
\[
\kappa=0 \quad \mbox{and} \quad V_1(\alpha)=V_3(\alpha)-\alpha=0, \quad V_2(V_2(\alpha)) =0, \quad V_2(\alpha) \neq 0,
\]
where $V_1,V_2$ and $V_3$ are the functions given in \eqref{basisA2ii}. In the former case, $\Sigma$ is always totally geodesic, in the latter case, it is never totally geodesic.
\vspace{1mm} 
\item[(3)] $\kappa=2$, $\gamma =0$, $\beta\neq 0$, $\delta\neq 0$, $\beta^2+d\delta^2 = \varepsilon$ and
\[
V_1(\beta)=V_2(\beta) = V_3(\beta) = 0,\quad V_1(\alpha) = 0,\quad V_2(\alpha)+\alpha\beta = 0,\quad V_3(\alpha)+\beta\delta =0,
\]
where $V_1,V_2$ and $V_3$ are the functions given in \eqref{basisA2iii}. These hypersurfaces are never totally geodesic.
\end{itemize}
\end{theorem} 

\begin{proof}
We consider cases (i), (ii) and (iii) listed in Theorem~\ref{codA2}. For case (i), with respect to the basis $\{V_1,V_2,V_3\}$  described in \eqref{basisA2i}, the shape operator is by \eqref{sA2i} given by
\begin{equation} \label{SA2i}
\left\{
\begin{array}{l}
SV_1 =  -\kappa\delta V_1, \\[2 pt]
SV_2 =  -\kappa\delta V_2, \\[2 pt]
SV_3 =  \varphi V_1 - \kappa\delta V_3,
\end{array}
\right.
\end{equation}
where $\varphi=\alpha^2- \varepsilon ^2-V_3(\alpha)$ and we took into account $V_1(\alpha)=V_2(\alpha)=0$. We then use \eqref{basisA2i} and \eqref{lambdaA2} to describe the connection $\nabla^{\Sigma}$ with respect to $\{V_1,V_2,V_3\}$. A long but straightforward calculation gives that the possibly non vanishing covariant derivatives $\nabla^\Sigma _{V_i} V_j$ are given by
\[
\begin{array}{ll}
\nabla^\Sigma _{V_2} V_2 = \frac{\kappa \alpha}{d\delta} V_1, \qquad &
\nabla^\Sigma _{V_2} V_3 = \kappa \alpha V_2, \\[6 pt]
\nabla^\Sigma _{V_3} V_1 = -(\kappa+1) \alpha V_1, \qquad &
\nabla^\Sigma _{V_3} V_3 = \psi V_1 +\alpha(\kappa+1) V_3,
\end{array}
\]
for a certain function $\psi$, which is unimportant for this computation. Then, $\Sigma$ is parallel if and only if $\nabla^\Sigma _{V_i} (SV_j)=S (\nabla^\Sigma _{V_i} V_j)$ for all indices $i,j$. By a direct calculation, this holds if and only if 
\[
V_1(\varphi) = V_2(\varphi) = 0,\quad V_3(\varphi) = 2(\kappa+1)\alpha \varphi.
\]
The first two equations follow from $V_1(\alpha)=V_2(\alpha)=0$. In particular, from \eqref{SA2i} we see that $\Sigma$ is totally geodesic if and only if $\kappa=\varphi=0$ and this completes case (1). 

Next, in case (ii), the shape operator is by \eqref{sA2ii} given by
\begin{equation*}
\left\{
\begin{array}{l}
SV_1 =  0, \\[2 pt]
SV_2 =  - V_2(\alpha)V_1, \\[2 pt]
SV_3 =  0,
\end{array}
\right.
\end{equation*}
%
so that $\Sigma$ is totally geodesic if and only if $V_2(\alpha)=0$. Next, \eqref{lambdaA2} yields that the connection $\nabla^{\Sigma}$ is completely determined by
\[
\begin{array}{ll}
\nabla^\Sigma _{V_1} V_2 = \nabla^\Sigma _{V_2} V_1 = \frac{\kappa \beta}{d} V_3, \qquad &
\nabla^\Sigma _{V_1} V_3 = \kappa V_1, \\[6 pt]
\nabla^\Sigma _{V_2} V_2 =-\frac{\alpha}{\beta}V_2(\alpha) V_1- \frac 1d \left(\kappa \alpha^2+(\kappa-1)\beta ^2 \right)V_3, \qquad &
\nabla^\Sigma _{V_2} V_3 = \beta V_1+\kappa V_2, \\[6 pt]
\nabla^\Sigma _{V_3} V_2 = \beta V_1+V_2,\qquad &
\nabla^\Sigma _{V_3} V_1 = -V_1, \\[6 pt]
\end{array}
\]

%
\vspace{1mm}\noindent
If $\Sigma$ is parallel, then 
%
%
\[
0 = S(\nabla^\Sigma_{V_2} V_2 ) = \nabla^\Sigma_{V_2} (SV_2) = \nabla^\Sigma_{V_2} (-V_2(\alpha) V_1) = -V_2(V_2(\alpha))V_1 - V_2(\alpha)\frac{\kappa \beta}{d} V_3.
\]
This implies, $V_2(\alpha)=0$ and either $\Sigma$ is totally geodesic or $V_2(V_2(\alpha))=0$ and  $\kappa = 0$. A quick check shows that under one of these conditions $\nabla^\Sigma_{V_i} (SV_j) = S(\nabla^\Sigma_{V_i} V_j)$ for all indices $i,j$.

Finally, in case (iii) the shape operator is by \eqref{sA2iii} given by
\[
\left\{
\begin{array}{l}
SV_1 =  -2\delta V_1, \\[2 pt]
SV_2 =  -(V_2(\alpha)+\alpha\beta) V_1  - \left(2\delta+\frac{V_2(\beta)}{d\delta} \right)V_2, \\[2 pt]
SV_3 =  -\left(V_3(\alpha) +\beta \delta \right) V_1 - \frac{V_3(\beta)}{d\delta} V_2 - 2\delta V_3,
\end{array}
\right.
\]
where we used the conditions $V_1(\alpha) = V_1(\beta)=0$ to simplify $S$. We then have the following components of the connection $\nabla^\Sigma$:
\begin{align*}
\nabla^\Sigma_{V_2} V_1 &= \beta V_1,\\
\nabla^\Sigma_{V_3} V_1 &= -2\varepsilon\alpha\beta\delta V_1-\frac{2\varepsilon \beta^2}{d}V_2,\\
\nabla^\Sigma_{V_3} V_2 &= \left(\frac{\varepsilon\alpha V_3(\beta) + 2d\alpha^2\delta^3}{\varepsilon\delta} - \beta^2\right)V_1  + \frac{2\alpha\beta\delta}{\varepsilon}V_2 - 2\beta V_3.
\end{align*}
%
A straightforward computation shows that $S(\nabla^\Sigma_{V_2} V_1) = \nabla^\Sigma_{V_2} S(V_1)$, $S(\nabla^\Sigma_{V_3} V_1 ) = \nabla^\Sigma_{V_3} S(V_1)$ and $S(\nabla^\Sigma_{V_3} V_2) = \nabla^\Sigma_{V_3} S(V_2)$ imply that the shape operator simplifies to
\[
\left\{
\begin{array}{l}
SV_1 =  -2\delta V_1, \\[2 pt]
SV_2 =  -2\delta V_2, \\[2 pt]
SV_3 =  -2\delta V_3,
\end{array}
\right.
\]
where $\delta$ is a constant. Clearly, this is also sufficient for $\Sigma$ to be parallel. Furthermore, we see that $\Sigma$ is totally geodesic if and only if we are either in case (i) or in case (ii).
\end{proof}

We shall now investigate in greater detail the examples we obtain under the assumption that functions  $\alpha,\beta,\gamma$ and $\delta$ are constant.

\begin{example}
	Suppose now that $\alpha,\beta,\gamma$ and $\delta$ are some real constants. We then  find the following examples of parallel hypersurfaces.
	\begin{itemize}
		\item If $\beta =\gamma = 0$, $d\delta^2=\varepsilon =\pm 1$ and $(\kappa+1)(\alpha^2-\varepsilon)\alpha=0$, then the images of the constant functions $V_1,V_2$ and $V_3$ span a Lorentzian (of signature either $(2,1)$ or $(1,2)$) subalgebra $\mathfrak{k}_1 \cong e(1,1)$ of $\g$. Thus, under these assumptions we find a  Lorentzian parallel hypersurface, which is an orbit of a Lorentzian Lie subgroup $K_1\subseteq G$ with Lie subalgebra $\mathfrak{k}_1$ and  satisfies conditions (1).
		\item If $\alpha=\gamma=\delta=0$ and $\beta^2=\varepsilon=1$, then the image of the constant functions $V_1,V_2$ and $V_3$ spans a Lorentzian subalgebra $\mathfrak{k}_2 \cong e(1,1)$ of $\g$. So, under these assumptions we find a Lorentzian parallel hypersurface, which is an orbit of a Lorentzian Lie subgroup $K_2\subseteq G$ with Lie subalgebra $\mathfrak{k}_2$ and satisfies conditions (2).
	\end{itemize}
\end{example}

%

\section{Hypersurfaces of homogeneous spaces of type \textbf{A3}}
\setcounter{equation}{0}
Let $F: \Sigma \to (M,g)$ be the immersion of a hypersurface into a non-reductive homogeneous pseudo-Riemannian four-manifold of type \textbf{A3}. As in the previous sections, we start by determining necessary algebraic conditions for hypersurfaces to have a Codazzi second fundamental form. 

\begin{lemma} \label{necA3}
Let $(M,g)$ be a non-reductive homogeneous pseudo-Riemannian four-manifold of type {\rm\bf A3} and $F : \Sigma \to M$ a hypersurface of $M$. Denote by $\xi=\alpha u_1+\beta u_2+\gamma u_3+\delta u_4$ a vector field normal to $\Sigma$, for some functions $\alpha, \beta, \gamma, \delta : G \to \R$, with $\langle \xi,\xi \rangle =\varepsilon=\pm 1$. The hypersurface $\Sigma$ has a Codazzi second fundamental form if and only if one of the following sets of conditions holds: 
\begin{itemize}
\item[(i)] $\gamma=\delta =0$;
\item[(ii)] $\beta=\delta=0$.
\end{itemize}
\end{lemma}
\begin{proof} Equation \eqref{gA3} yields that  
\[\langle X,\xi\rangle=a\delta x_1+a \beta x_2+(b\gamma +c\delta)x_3+(a\alpha +c\gamma+d\delta) x_4,\]
for an arbitrary function $X=x_1u_1+x_2u_2+x_3u_3+x_4u_4$. Thus, 
\begin{equation}\label{nxiA3}
\langle \xi,\xi\rangle=2a \alpha\delta +a\beta^ 2+b\gamma ^2 +2c\gamma \delta +d\delta^2 = \varepsilon
\end{equation}
and $X$ is perpendicular to $\xi$ if and only if 
\[
\langle X,\xi\rangle=a\delta x_1+a \beta x_2+(b\gamma +c\delta)x_3+(a\alpha +c\gamma) x_4=0.
\]
Let $X=\beta u_1 - \delta u_2$ and $Y=(a\alpha + c\gamma + d\delta)u_2 - a\beta u_4$. Then $X$ and $Y$ are perpendicular to $\xi$. Requiring $R(X,Y)\xi=0$, by \eqref{RA3} we find that either $\beta=0$ or $\delta=0$, where we excluded the case of constant sectional curvature $d+\eta a  = 0$.

Suppose $\beta =0$ and let $Z = a\delta u_4 - (a\alpha+c\gamma+d\delta)u_1$. Then $Z$ is perpendicular to $\xi$ and $R(X,Z)\xi = 0$ implies $\delta = 0$, which gives case (ii).

Suppose $\beta \neq 0$, then automatically $\delta = 0$. Let $W = a\beta u_3- (b\gamma+c\delta)u_2$, this is perpendicular to $\xi$ and $R(Y,W)\xi = 0$ implies $\gamma = 0$, which gives case (i). 

For the converse, remark that the following functions span the orthogonal complement of $\xi(g)$ for every $g\in G$:
\begin{align}
& \mbox{Case (i)}: & V_1=u_1, &&& V_2=u_3, && V_3=\alpha u_2 - \beta u_4, \label{basisA3i} \\
& \mbox{Case (ii)}: & V_1=u_1, &&& V_2=u_2, && V_3=(a\alpha+c\gamma)u_3 -b\gamma u_4, \label{basisA3ii} 
\end{align}
It is now sufficient to check that, in each of the three cases, $R(V_i,V_j)\xi =0$ for all $i,j \in \{1,2,3\}$.
\end{proof}
%

\begin{theorem}\label{codA3}
In the same notations of Lemma~{\em \ref{necA3}}, the hypersurface $\Sigma$ has Codazzi second fundamental form if and only if one of the following sets of conditions holds: 
\begin{itemize}
\item[(i)] $\gamma=\delta=0$, $a\beta^2 = \varepsilon$ and  \[
V_1(\alpha)=V_2(\alpha)-\alpha=0, 
\]
where $\{V_1,V_2,V_3\}$ is described in \eqref{basisA3i}.
\vspace{1mm}
\item[(ii)] $\beta = \delta = 0$, $b\gamma^2 = \varepsilon$ and 
\[
V_1(\alpha)=V_2(\alpha)=0,
\]
where $\{V_1,V_2,V_3\}$ is described in \eqref{basisA3ii}.
\end{itemize}
\end{theorem}

\begin{proof}
We first consider $\gamma=\delta =0$. Equation \eqref{nxiA3} now yields $a\beta ^2 =\varepsilon$ and so, $\beta \neq 0$ is a constant. By \eqref{lambdaA3}, we obtain the shape operator with respect to the basis \eqref{basisA3i}:
\begin{equation}\label{sA3i}
\left\{
\begin{array}{l}
SV_1 =  -V_1(\alpha) V_1, \\[2 pt]
SV_2 =  -(V_2(\alpha)-\alpha) V_1, \\[2 pt]
SV_3 =  -(V_3(\alpha)-\alpha\beta)V_1.
\end{array}
\right.
\end{equation}
%
Hence, $S$ is self-adjoint if and only if $V_1(\alpha)=V_2(\alpha)-\alpha=0$. This is case (i).

Suppose that $\beta=\delta=0$. Then, \eqref{nxiA3} yields $b\gamma^2 = \varepsilon$ and so, $\gamma \neq 0$ is a constant. 
Using the description of the Levi-Civita connection given in \eqref{lambdaA3}, we find that the shape operator $S$ with respect to the basis from \eqref{basisA3ii} is as follows:
\begin{equation}\label{sA3ii}
\left\{
\begin{array}{l}
SV_1 =  -(V_1(\alpha)+\gamma) V_1, \\[2 pt]
SV_2 =  -V_2(\alpha)V_1 - \gamma V_2, \\[2 pt]
SV_3 =  \left(-V_3(\alpha) + a\alpha^2+2c\alpha\gamma + \frac{c^2\gamma^2-bd\gamma^2}{a}\right)V_1 - \gamma V_3.
\end{array}
\right.
\end{equation}
%
%
Hence, $S$ is self-adjoint if and only if $V_1(\alpha)=V_2(\alpha)=0$. This is case (ii).
\end{proof}

We can now classify parallel and totally geodesic hypersurfaces of a non-reductive homogeneous pseudo-Riemannian space of type \textbf{A3}. 

\begin{theorem}\label{parA3}
Let $(M,g)$ be a non-reductive homogeneous pseudo-Riemannian four-manifold of type {\rm\bf A3} and $F : \Sigma \to M$ a hypersurface of $M$. Denote by $\xi=\alpha u_1+\beta u_2+\gamma u_3+\delta u_4$ a vector field normal to $\Sigma$, with $\langle \xi ,\xi \rangle=\varepsilon=\pm 1$. Then, $\Sigma$ is a parallel hypersurface if and only if one of the following sets of conditions holds:
\begin{itemize}

\item[(1)]  $\gamma=\delta=0$, $a\beta^2=\varepsilon$ and
\[
V_1(\alpha)=V_2(\alpha)-\alpha=0,\quad V_3(\alpha) = 0
\]
where $\{V_1,V_2,V_3\}$ is described in \eqref{basisA3i}. In particular, these hypersurfaces are always totally geodesic.
\vspace{1mm} 
\item[(2)] $\beta=\delta=0$, $b\gamma^2=\varepsilon$ and
\[
V_1(\alpha)=V_2(\alpha)=0,\quad V_3(\varphi) = 4(a\alpha+c\gamma)\varphi
\]
where $\varphi=- V_3(\alpha)+a\alpha^2+2c\alpha\gamma +\frac{c^2-bd}{a}\gamma^2$ and $\{V_1,V_2,V_3\}$ is described in \eqref{basisA3ii}. These hypersurfaces are never totally geodesic.
\end{itemize}
\end{theorem} 
\begin{proof}
We consider cases (i) and (ii) listed in Theorem~\ref{codA3}. For case (i), with respect to the basis $\{V_1,V_2,V_3\}$  described in \eqref{basisA3i}, the shape operator in \eqref{sA3i} is given by
\[
\left\{
\begin{array}{l}
SV_1 =  0, \\[2 pt]
SV_2 =  0, \\[2 pt]
SV_3 =  -(V_3(\alpha)-\alpha\beta)V_1.
\end{array}
\right.
\]
Next, we use \eqref{lambdaA3} to describe the connection $\nabla^{\Sigma}$ with respect to $\{V_1,V_2,V_3\}$. The possibly nonvanishing covariant derivatives $\nabla^\Sigma _{V_i} V_j$ are given by
\[
\begin{array}{ll}
\nabla^\Sigma _{V_1} V_2 =-\nabla^\Sigma _{V_2} V_1 =  V_1, \qquad &
\nabla^\Sigma _{V_1} V_3 = \nabla^\Sigma _{V_3} V_1 = -\frac cb \beta V_1 +\frac ab \beta V_2 , \\[8 pt]
\nabla^\Sigma _{V_2} V_3 = \nabla^\Sigma _{V_3} V_2 = \frac{bd-c^2}{ab} \beta V_1 +\frac cb \beta V_2 +V_3, \qquad &
\nabla^\Sigma _{V_3} V_3 = \left( \frac cb \alpha ^2 -\frac{\alpha}{\beta} V_3(\alpha) \right) V_1-\frac ab \alpha ^2 V_2 .\\[4 pt]
\end{array}
\]
It is now easy to check that $\nabla^\Sigma _{V_i} SV_j=S (\nabla^\Sigma _{V_i} V_j)$ for all indices $i,j$ if and only if $V_3(\alpha)=\alpha\beta$. This implies $S=0$ and so, $\Sigma$ is totally geodesic.

Next, in case (ii), the shape operator in \eqref{sA3ii} is described by
\[
\left\{
\begin{array}{l}
SV_1 =  -\gamma V_1, \\[2 pt]
SV_2 =  -\gamma V_2, \\[2 pt]
SV_3 =  -\varphi V_1 - \gamma V_3,
\end{array}
\right.
\]
the possibly nonvanishing covariant derivatives $\nabla^\Sigma _{V_i} V_j$ are given by 
\[
\begin{array}{ll}
\nabla^\Sigma _{V_2} V_2 = \left(\frac{a\alpha+c\gamma}{b\gamma} - 1 \right) V_1, \quad &
\nabla^\Sigma _{V_2} V_3 = (a\alpha+c\gamma-b\gamma)V_2,\\[6 pt]
\nabla^\Sigma _{V_3} V_1 = -2(a\alpha+c\gamma)V_1,\quad &
\nabla^\Sigma _{V_3} V_3 = \frac{-a^2\alpha V_3(\alpha) + 2(a\alpha+c\gamma)(a^2\alpha^2+(bd-c^2)\gamma^2)}{a\gamma})V_1 + 2(a\alpha + c\gamma)V_3.
\end{array}
\]
It is now easy to check that $\nabla^\Sigma _{V_i} SV_j=S (\nabla^\Sigma _{V_i} V_j)$ for all indices $i,j$ if and only if $V_3(\varphi) = 4(a\alpha+c\gamma)\varphi$. The hypersurface is never totally geodesic. 
\end{proof}

\begin{example}
	Suppose $\alpha,\beta,\gamma$ and $\delta$ are constant functions. Then we find the following parallel hypersurfaces.
	\begin{itemize}
		\item If $\alpha =\gamma = \delta = 0$ and $a\beta^2=\varepsilon$, then the image of the constant functions $V_1,V_2$ and $V_3$ spans a subalgebra  $\mathfrak{k}_1 \cong \mathbb{R} \times \mathfrak{s}(2)$ of $\g$, with $\mathfrak{s}(2)$ the solvable, non-abelian two-dimensional Lie algebra. Thus under these assumptions we find a Lorentzian parallel hypersurface, which is an orbit of a Lorentzian Lie subgroup $K_1 \subseteq G$ with Lie subalgebra $\mathfrak{k}_1$ and satisfies conditions (1).
		\item If $\beta = \delta = 0$, $b\gamma^2=\varepsilon$ and $(a\alpha + c\gamma)\varphi = 0$,  then the image of the constant functions $V_1,V_2,V_3$ together with the isotropy algebra span a pseudo-Riemannian subalgebra  
		$\mathfrak{k}_2 \cong e(1,1)$ of $\g$. So, under these assumptions we find a Lorentzian parallel hypersurface, which is an orbit of a Lorentzian Lie subgroup $K_2 \subseteq G$ with Lie subalgebra $\mathfrak{k}_2$ and satisfies conditions (2).
	\end{itemize}
\end{example}

%

\section{Hypersurfaces of homogeneous spaces of type \textbf{A4} and \textbf{B2}}
\setcounter{equation}{0}

\begin{lemma} \label{necA4}
Let $(M,g)$ be a non-reductive homogeneous pseudo-Riemannian four-manifold of type {\rm\bf A4} or {\rm\bf B2} and $F : \Sigma \to M$ a hypersurface of $M$. Denote by $\xi=\alpha u_1+\beta u_2+\gamma u_3+\delta u_4$ a vector field normal to $\Sigma$, for some functions $\alpha, \beta, \gamma, \delta : M \to \R$, with $\langle \xi,\xi \rangle =\varepsilon=\pm 1$. The hypersurface $\Sigma$ has a Codazzi second fundamental form if and only if one of the following sets of conditions holds: 
\begin{itemize}
\item[(i)] $\alpha=\beta=0$; 
\item[(ii)]  $\beta=\delta=0$.
\end{itemize}
\end{lemma}
\begin{proof} For an arbitrary function $X=x_1u_1+x_2u_2+x_3u_3+x_4u_4$, from \eqref{gA4} we get 
\[\langle X,\xi\rangle= \eta a \alpha x_1+(a\gamma +b\beta)x_2+ a\beta x_3+\frac 12 a\delta x_4.\]
Hence,
\begin{equation}\label{nxiA4}
\langle \xi,\xi\rangle= \eta a \alpha^2 +2a\beta\gamma +b\beta^2+\frac 12 a \delta^2=\varepsilon
\end{equation}
and $X$ is orthogonal to $\xi$ if and only if 
\[\begin{array}{l}
\eta a \alpha x_1+(a\gamma +b\beta)x_2+ a\beta x_3+\frac 12 a\delta x_4=0.
\end{array}
\]
Considering the functions $X=\beta u_1 - \eta\alpha u_3$ and $Y=(a\gamma+b\beta) u_1 - \eta a\alpha u_2$ perpendicular to $\xi$. From  \eqref{RA4} we get that $R(X,Y)\xi=0$ if and only if either $\alpha=0$ or $\beta=0$.  

However, if $\alpha=0$, then $R(u_1,a\beta u_2-(a\gamma+b\beta)u_3)\xi=0$ yields again $\beta=0$. Hence, necessarily $\beta=0$. Next, taking $Z=\gamma u_1 -\eta\alpha u_2$ and $W=\delta u_1-2\eta\alpha u_4$, we find that $R(Z,W)\xi=0$ if and only if either $\alpha=0$ or $\delta=0$. This gives cases (i) and (ii) listed above.

For the converse, remark that the following functions are perpendicular to $\xi(g)$ for all $g\in G$:
\begin{align}
& \mbox{Case (i)}: & V_1=u_1, &&& V_2=\delta u_2-2\gamma u_4, && V_3=u_3, \label{basisA4i} \\
& \mbox{Case (ii)}: & V_1=u_4, &&& V_2=\gamma u_1-\eta\alpha u_2, && V_3=u_3, \label{basisA4ii} 
\end{align}
It is now sufficient to check that, in each of the three cases, $R(V_i,V_j)\xi =0$ for all $i,j \in \{1,2,3\}$.
\end{proof}
Next, we determine differential conditions on the functions $\alpha,\beta,\gamma$ and $\delta$ for hypersurfaces of a space of type \textbf{A4} and \textbf{B2} to have a Codazzi second fundamental form.

\begin{theorem}\label{codA4}
In the same notations of Lemma~{\em \ref{necA4}}, the hypersurface $\Sigma$ has Codazzi second fundamental form if and only if one of the following sets of conditions holds: 
\begin{itemize}
\item[(i)] $\alpha=\beta=0$, $a\delta^2 = 2\varepsilon$ and
\begin{equation}\label{intA4i}
V_1(\gamma)-\gamma = V_3(\gamma)= 0,
\end{equation}
where $V_1,V_2$ and $V_3$ are the functions described in \eqref{basisA4i}; 
\item[(ii)] $\beta=\delta=0$, $a \alpha^2 = \varepsilon \eta$ and
\begin{equation}\label{intA4ii}
V_1(\gamma) = V_3(\gamma) = 0,
\end{equation}
where $V_1,V_2$ and $V_3$ are the functions described in \eqref{basisA4ii}. 
\end{itemize}
\end{theorem} 

\begin{proof}
We treat the three cases listed in Lemma~\ref{necA4} separately. 

\smallskip
{\em Case }(i). In this case, \eqref{nxiA4} implies $a\delta^2=2\varepsilon$, so that $\delta \neq 0$ is constant. Using \eqref{basisA4i} and \eqref{lambdaA4}, we calculate the shape operator $S$ and we find
\begin{equation}\label{sA4i}
\left\{
\begin{array}{l}
SV_1 =  -(V_1(\alpha)+\gamma) V_3, \\[2 pt]
SV_2 =  -V_2(\gamma) V_3, \\[2 pt]
SV_3 =  -V_3(\gamma)V_3.
\end{array}
\right.
\end{equation}
%
%
%
%
We then use \eqref{basisA4i} and \eqref{gA4} to decide when $S$ is self-adjoint and we find that this holds if and only if \eqref{intA4i} is satisfied. 

\smallskip
{\em Case} (ii). We now have $\beta=\delta=0$. Thus, $a\alpha ^2 =\varepsilon \eta$ and so, $\alpha$ is constant and non-zero. Using \eqref{basisA4ii} and \eqref{lambdaA4}, we calculate the shape operator $S$ and we find
\begin{equation}\label{sA4ii}
\left\{
\begin{array}{l}
SV_1 =  \alpha V_1 -V_1(\gamma) V_3, \\[2 pt]
SV_2 =  \alpha V_2 + \left(-V_2(\gamma)+\gamma^2-\frac{\eta\alpha^2 b }{a}\right) V_3, \\[2 pt]
SV_3 =  -(V_3(\gamma)-\alpha)V_3.
\end{array}
\right.
\end{equation}
Hence, $S$ is self-adjoint if and only if equation \eqref{intA4ii} is satisfied.
This completes case (ii). 
\end{proof}

With regard to parallel and totally geodesic hypersurfaces of a non-reductive homogeneous pseudo-Riemannian space of type \textbf{A4} and \textbf{B2}, we obtain the following.

\begin{theorem}\label{parA4}
Let $(M,g)$ be a non-reductive homogeneous pseudo-Riemannian four-manifold of type {\rm\bf A4} or {\rm\bf B2} and $F : \Sigma \to M$ a hypersurface of $M$. Denote by $\xi=\alpha u_1+\beta u_2+\gamma u_3+\delta u_4$ a vector field normal to $\Sigma$, with $\langle \xi ,\xi \rangle=\varepsilon=\pm 1$. Then, $\Sigma$ is a parallel hypersurface if and only if
\begin{itemize}
\item[(1)] $\alpha=\beta=0$, $a\delta^2 = 2\varepsilon$ is constant and 
\begin{equation}\label{parA4i}
\begin{array}{l}
V_1(\gamma)-\gamma=V_2(\gamma)=V_3(\gamma)=0,
\end{array}
\end{equation}
where $V_1,V_2$ and $V_3$ are the functions described in \eqref{basisA4i}. In particular, these hypersurfaces are always totally geodesic. 
\item[(2)] $\beta = \delta= 0$, $a\alpha^2 = \varepsilon \eta$ is constant  and
\begin{equation}\label{parA4ii}
V_1(\gamma) = V_3(\gamma) = V_2(\varphi) = 0,
\end{equation}
where $\varphi = -V_2(\gamma) + \gamma^2-\frac{\alpha^2b}{a}\eta$ and $V_1,V_2$ and $V_3$ are the functions described in \eqref{basisA4ii}. Moreover, these hypersurfaces are always proper parallel. 
\end{itemize}
\end{theorem} 

\begin{proof}
We start from cases (i), (ii) listed in Theorem~\ref{codA4}. In case (i), the description of the shape operator we gave within the proof of Theorem~\ref{codA4} implies that $S=0$ if and only if $V_2(\gamma)=0$. From \eqref{basisA4i} and \eqref{lambdaA4} we find that the possibly nonvanishing covariant derivatives $\nabla^\Sigma _{V_i} V_j$ are given by
$$
\begin{array}{ll}
\nabla^\Sigma _{V_1} V_2 = V_2 -\frac{b\delta}{a}V_3, \qquad &
\nabla^\Sigma _{V_1} V_3 = \nabla^\Sigma _{V_3} V_1 = -V_3 , \\[6 pt]
\nabla^\Sigma _{V_2} V_1 = -V_2 - \frac{b\delta}{a}V_3,\qquad &
\nabla^\Sigma _{V_2} V_2 = \left(\frac{2\gamma^2}{\varepsilon} + \frac{2 \delta^2 b}{a\varepsilon}\right)	V_1 + \left(\frac{2V_2(\gamma)\gamma}{\delta}\right)V_3, \\[6 pt]
\nabla^\Sigma _{V_2} V_3 = \nabla^\Sigma _{V_3} V_2= \frac{\delta}{\varepsilon}V_1. \qquad &
\end{array}
$$
If $\Sigma$ is parallel then
\begin{align*}
0 &= S(\nabla^\Sigma_{V_1}V_2) - \nabla^\Sigma_{V_1}S(V_2)= -V_2(\gamma)V_3 - (-V_1(V_2(\gamma)) +V_2(\gamma))V_3\\
  &= (V_1(V_2(\gamma))-2V_2(\gamma))V_3 = (V_2(V_1(\gamma)) + [V_1,V_2](\gamma) - 2V_2(\gamma))V_3\\
  &=V_2(\gamma)V_3.
\end{align*}
From \eqref{sA4i} we conclude that $\Sigma$ is parallel if and only if $\Sigma$ is totally geodesic.

With regard to case (ii), using \eqref{basisA4ii} and \eqref{lambdaA4} we get that the possibly nonvanishing covariant derivatives $\nabla^\Sigma _{V_i} V_j$ are given by
\[
\begin{array}{l}
\nabla^\Sigma _{V_1} V_1 = -\frac{\gamma\eta}{2\alpha} V_3, \\[4 pt]
\nabla^\Sigma _{V_1} V_2 = -\gamma V_1, \\[4 pt]
\nabla^\Sigma _{V_2} V_2 = -\frac{V_2(\gamma)\gamma}{\alpha}V_3.
\end{array}
\]
The shape operator in \eqref{sA4ii} is now given by
\[
\left\{
\begin{array}{l}
SV_1 =  \alpha V_1, \\[2 pt]
SV_2 =  \alpha V_2 + \varphi V_3, \\[2 pt]
SV_3 =  \alpha V_3,
\end{array}
\right.
\]
where $\varphi  = -V_2(\gamma) + \gamma^2-\frac{\alpha^2b}{a}\eta$. Then, $\nabla^{\Sigma} _{V_i}  SV_j=S \nabla^{\Sigma} _{V_i}  V_j$ holds if and only if 
$V_1(\varphi) = V_2(\varphi) = V_3(\varphi) = 0$.
It is easy to check that $V_1(\varphi) = V_3(\varphi) = 0$ already follow from $V_1(\gamma)=V_3(\gamma) = 0$. Hence, case 2 follows.
\end{proof}

\begin{example}
	Suppose $\alpha,\beta,\gamma$ and $\delta$ are constant functions. Then we find the following parallel hypersurfaces.
	\begin{itemize}
		\item If $\alpha =\beta =\gamma= 0$, $a\delta^2=2\varepsilon$, then the image of the constant functions $V_1,V_2$ and 
		$V_3$ spans a subalgebra $\mathfrak{k}_1 \cong \mathbb{R} \times \mathfrak{s}(2)$ of $\g$. Thus, under these assumptions we find a Lorentzian parallel hypersurface, which is an orbit of a Lorentzian Lie subgroup 
$K_1 \subseteq G$ with Lie subalgebra $\mathfrak{k}_1$ and satisfies conditions (1).
		\item If $\beta=\delta=0$ and $a\alpha^2=\varepsilon\eta$, then the image of the constant functions $V_1,V_2$ and $V_3$ spans a subalgebra $\mathfrak{k}_2 \cong \mathbb{R} \times \mathfrak{s}(2)$ of $\g$. So, under these assumptions we find a Lorentzian parallel hypersurface, which is an orbit of a Lorentzian Lie subgroup $K_2 \subseteq G$ with Lie subalgebra $\mathfrak{k}_2$ and satisfies conditions (2).
	\end{itemize}
\end{example}



\section{Hypersurfaces of homogeneous spaces of type \textbf{B1}}
\setcounter{equation}{0}

\begin{lemma} \label{necB1}
Let $(M,g)$ be a homogeneous pseudo-Riemannian four-manifold of type {\rm\bf B1} with metric~\eqref{gA1}. If there exists a non-degenerate hypersurface $F : \Sigma \to M$ with Codazzi second fundamental form, then $b=c=d=0$.
\end{lemma}
\begin{proof} By equation \eqref{gB1}, for an arbitrary function $X=x_1u_1+x_2u_2+x_3u_3+x_4u_4$, we have 
\[
\langle X,\xi\rangle=a\gamma x_1+(a \delta +b\beta +c\gamma ) x_2+(a \alpha +c\beta +d\gamma )x_3+a \beta x_4.
\]
Hence,
\begin{equation*}\label{nxiB1}
\langle \xi,\xi\rangle=b\beta^2 +2a\alpha \gamma +2a\beta\delta +2c\beta \gamma +d\gamma^2 = \varepsilon
\end{equation*}
and $X$ is perpendicular to $\xi$ if and only if 
\[
a\gamma x_1+(a \delta +b\beta +c\gamma ) x_2+(a \alpha +c\beta +d\gamma )x_3+a \beta x_4=0.
\]

First suppose that $\beta \neq 0$ then $ X = \beta u_1 - \gamma u_4,\ Y = a\beta u_2 -(a\delta + b\beta +c\gamma)u_4$ and $Z = a \beta u_3 - (a\alpha + c\beta +d\gamma)u_4$ are perpendicular to $\xi$. From \eqref{RB1} we obtain $R(X,Z)\xi = 0$ if and only if $c=d=0$. This then implies $R(X,Y)\xi = 0$ if and only if $b=0$. 

Assume now that $\beta = 0$. In this case, $\langle \xi, \xi\rangle =\gamma(2a\alpha+d\gamma)$ implies $\gamma\neq 0$. Let $X = (a\alpha+d\gamma)u_1 -a\gamma u_3 ,\  Y = (a\delta+c\gamma)u_1-a\gamma u_2$ and $Z = u_4$. These functions span a basis of $\xi^\perp(g)$ for every $g\in G$. In this case $R(Y,Z)\xi = 0$ implies $d=0$. Then $\left\langle\xi,\xi\right\rangle = 2 a \alpha\gamma$ implies that also $\alpha \neq 0$. For $\alpha\neq 0$ we have $R(X,Z)\xi = R(Y,Z)\xi = 0$ if and only if $c=d=0$. The equation $R(X,Y)\xi = 0$ now implies $b=0$. 
\end{proof}
By Remark~\ref{rem:symmetric cases}, a homogeneous space of type \textbf{B1} with $b=c=d=0$ has constant sectional curvature. In particular, it is not non-reductive and hence we have the following.
\begin{theorem}
	A non-reductive homogeneous pseudo-Riemannian four-manifold of type {\rm\bf B1} does not allow hypersurfaces with Codazzi second fundamental form. In particular, it does not allow parallel and totally geodesic hypersurfaces.
\end{theorem}

\bibliographystyle{amsalpha_abbrv}
\bibliography{./ParHyp4Dnotred.bib}

\end{document}